\documentclass{lms}

\usepackage{amsrefs}
\usepackage{amsmath,amssymb,mathrsfs}
\usepackage[all]{xy}
\newdir{ >}{{}*!/-5pt/@{>}} 
\SelectTips{cm}{} 

\newtheorem{theorem}{Theorem}[section] 
\newtheorem{lemma}[theorem]{Lemma}     
\newtheorem{proposition}[theorem]{Proposition}

\newnumbered{remark}[theorem]{Remark}

\title[The circle action on $THH$ of $MU$ and $BP$]
 {The circle action on topological Hochschild homology \\
	of complex cobordism and the Brown--Peterson spectrum}

\author{John Rognes}

\classno{55N22 (primary), 55P43, 55P91 (secondary)}

\providecommand{\holim}{\mathop{\rm holim}}
\providecommand{\cok}{\mathop{\rm cok}\nolimits}

\providecommand{\Tor}{\mathop{\rm Tor}\nolimits}
\newcommand{\cG}{\mathcal{G}}
\newcommand{\cT}{\mathcal{T}}
\newcommand{\bC}{\mathbb{C}}
\newcommand{\bF}{\mathbb{F}}
\newcommand{\bQ}{\mathbb{Q}}
\newcommand{\bZ}{\mathbb{Z}}
\newcommand{\longto}{\longrightarrow}
\newcommand{\longfrom}{\longleftarrow}
\renewcommand{\:}{\colon}

\allowdisplaybreaks

\begin{document}
\maketitle

\begin{abstract}
We specify exterior generators
in $\pi_* THH(MU) = \pi_*(MU) \otimes E(\lambda'_n \mid n\ge1)$
and $\pi_* THH(BP) = \pi_*(BP) \otimes E(\lambda_n \mid n\ge1)$,
and calculate the action of the $\sigma$-operator on these
graded rings.  In particular, $\sigma(\lambda'_n) = 0$ and
$\sigma(\lambda_n) = 0$, while the actions on $\pi_*(MU)$ and
$\pi_*(BP)$ are expressed in terms of the right units $\eta_R$
in the Hopf algebroids $(\pi_*(MU), \pi_*(MU \wedge MU))$
and $(\pi_*(BP), \pi_*(BP \wedge BP))$, respectively.
\end{abstract}

\section{Introduction}

Let $S$ be the sphere spectrum.  For any (associative) $S$-algebra
$R$, the topological Hoch\-schild homology spectrum $THH(R)$ is the
geometric realization of a cyclic spectrum $[q] \mapsto THH(R)_q = R
\wedge R^{\wedge q}$, see~\cite{BHM93} and~\cite{EKMM97}.  The skeleton
filtration of $THH(R)$ leads to a spectral sequence
$$
E^1_{q,*} = \pi_{q+*}(sk_q THH(R), sk_{q-1} THH(R))
	\Longrightarrow \pi_{q+*} THH(R) \,,
$$
whose $(E^1, d^1)$-term is the normalized chain complex associated to
the simplicial graded abelian group
$$
[q] \mapsto \pi_* THH(R)_q = \pi_*(R \wedge R^{\wedge q}) \,.
$$
The cyclic structure specifies a natural circle action on $THH(R)$, which
we shall treat as a right action.  The cofiber sequence $1_+ \to S^1_+
\to S^1$ is split by a retraction $S^1_+ \to 1_+$ and a stable section
$S^1 \to S^1_+$.  We write $\sigma$ for the composite map $THH(R) \wedge S^1
\to THH(R) \wedge S^1_+ \to THH(R)$ and call the induced homomorphism
$\sigma \: \pi_* THH(R) \to \pi_{*+1} THH(R)$ the (right)
$\sigma$-operator.  It satisfies $\sigma^2 = \eta \sigma$, where $\eta
\in \pi_1(S)$ is the complex Hopf map, so if multiplication by $\eta$
acts trivially on $\pi_* THH(R)$ then $\sigma$ is a differential.

There is a spectral sequence
\begin{equation} \label{eq:TPspseq}
E^2_{*,*} = \hat H^{-*}(S^1; \pi_* THH(R))
	= \bZ[t, t^{-1}] \otimes \pi_* THH(R)
	\Longrightarrow \pi_* THH(R)^{tS^1}
\end{equation}
converging to the homotopy of the circle Tate construction on $THH(R)$,
see~\cite{GM95}, more recently known~\cite{Hes18} as the periodic
topological cyclic homology $\pi_* TP(R)$.  Its initial differential is
given by
$$
d^2(t^n \cdot x) = \begin{cases}
t^{n+1} \cdot \sigma(x) & \text{for $n$ even} \\
t^{n+1} \cdot (\sigma(x) + \eta x) & \text{for $n$ odd.}
\end{cases}
$$
Knowledge of the $\sigma$-operator therefore leads to knowledge of
the $E^3 = E^4$-term of this spectral sequence.
When $\eta$ act trivially on $\pi_* THH(R)$ we can write
$$
E^4_{*,*} = \bZ[t, t^{-1}] \otimes H(\pi_* THH(R), \sigma) \,.
$$
In this paper we determine the $\sigma$-operator on $\pi_* THH(MU)$
and $\pi_* THH(BP)$, where $MU$ is the complex cobordism $E_\infty$
ring spectrum~\cite{Mil60}, \cite{Nov60}, \cite{May77} and $BP$ is
the Brown--Peterson $E_4$ ring spectrum~\cite{BP66}, ~\cite{BM13}.
In these cases $THH(R)$ is an $E_\infty$, resp.~$E_3$, ring spectrum
by \cite{BFV07}, $\sigma$ is a (right) derivation by \cite{AR05}, and
the skeleton and Tate spectral sequences are algebra spectral sequences
\cite{HR}.

In Sections~\ref{sec:fglamc} and~\ref{sec:typ} we review the connection
between complex cobordism and formal group laws, and their $p$-typical
variants, including some explicit formulas in the Hopf algebroids
$$
(\pi_*(MU), \pi_*(MU \wedge MU)) \cong (L, LB) \cong (L, LC)
$$
and
$$
(\pi_*(BP), \pi_*(BP \wedge BP)) \cong (V, VT) \,.
$$
We follow the expositions by Adams~\cite{Ada74} and
Landweber~\cite{Lan73},~\cite{Lan75} of Quillen's theory~\cite{Qui69},
adding some less familiar details about the parametrization of strict
isomorphisms of formal group laws by ``moving coordinates'' using $(L,
LC)$, in place of ``absolute coordinates'' using $(L, LB)$.

In Section~\ref{sec:thhatbc} we obtain isomorphisms of simplicial
commutative rings
$$
\pi_* THH(MU)_\bullet \cong \pi_*(MU) \otimes \beta(B)_\bullet
	\cong \pi_*(MU) \otimes \beta(C)_\bullet \,,
$$
in the spirit of the equivalence $THH(MU) \simeq MU \wedge BBU_+$ of
Blumberg, Cohen and Schlicht\-krull~\cite{BCS10}.  Here $\beta(B)_\bullet$
denotes the simplicial bar construction $[q] \mapsto \beta(B)_q =
B^{\otimes q}$, and similarly for $\beta(C)_\bullet$.  We also obtain
analogous information for $\pi_* THH(BP)_\bullet$.

In Section~\ref{sec:tcaatru} we recognize the circle action on
the $0$-simplices in $THH(MU)$ and $THH(BP)$ as being given by
the right units $\eta_R \: L \to LB \cong LC$ and $\eta_R \: V
\to VT$, respectively, and use this to determine the action of the
$\sigma$-operator on $\pi_* THH(MU)$ and $\pi_* THH(BP)$.  More precisely,
in Proposition~\ref{prop:sigma-etaR} we prove that for $x \in \pi_*(R)$
the homotopy class $\sigma(x) \in \pi_{*+1} THH(R)$ is detected in
$E^\infty_{1,*}$ of the skeleton spectral sequence by the class of $(1
\wedge \pi) \eta_R(x) \in \pi_*(R \wedge R/S) = E^1_{1,*}$.  Here $\eta_R
\: R \cong S \wedge R \to R \wedge R$ and $\pi \: R \to R/S$ are the
evident maps.

According to McClure and Staffeld \cite{MS93}, who credit Andy Baker
and Larry Smith, there are isomorphisms
$$
\pi_* THH(MU) \cong \pi_*(MU) \otimes E(\lambda'_n \mid n\ge1)
$$
with $\lambda'_n$ in degree $2n+1$, and
$$
\pi_* THH(BP) \cong \pi_*(BP) \otimes E(\lambda_n \mid n\ge1)
$$
with $\lambda_n$ in degree $2p^n-1$, at each prime~$p$.  We
strengthen these results, in Theorems~\ref{thm:sigma-Llambdaprimen}
and~\ref{thm:sigma-Vlambdan}, to show that the exterior generators
$\lambda'_n$ and $\lambda_n$ can be chosen so that $\sigma(\lambda'_n)
= 0$ and $\sigma(\lambda_n) = 0$, for all $n\ge1$.  These choices are
naturally connected to the moving coordinates on strict isomorphisms
between formal group laws, or $p$-typical formal group laws, as made
precise in Propositions~\ref{prop:skelTHHMU} and~\ref{prop:skelTHHBP}.
On the other hand, we show in Theorem~\ref{thm:sigma-Len} that
$\sigma(e_3)$ and $\sigma(e_4)$ are nonzero for the alternative
sequence of exterior generators $e_n$ of $\pi_* THH(MU)$ associated,
as in Proposition~\ref{prop:skelTHHMUsplit}, to absolute coordinates.

We can summarize Proposition~\ref{prop:skelTHHBP},
Theorem~\ref{thm:sigma-Vlambdan} and equations~\eqref{eq:Hazewinkel}
and~\eqref{eq:sigmavn-recursive} as follows.

\begin{theorem}
Let $\pi_*(BP) = \bZ_{(p)}[v_n \mid n\ge1]$ where the $v_n$
are the Hazewinkel generators.
The $\sigma$-operator $\sigma \: \pi_* THH(BP) \to \pi_{*+1} THH(BP)$
is a (right) derivation acting on
$$
\pi_* THH(BP) = \pi_*(BP) \otimes E(\lambda_n \mid n\ge1) \,.
$$
It satisfies $\sigma(\lambda_n) = 0$ for all $n\ge1$, while
$\sigma(v_n)$ is recursively determined by the equation
$$
p \lambda_n = \sigma(v_n) + \sum_{i=1}^{n-1} \Bigl(v_{n-i}^{p^i} \lambda_i
        + (p^i \ell_i) v_{n-i}^{p^i-1} \sigma(v_{n-i}) \Bigr) \,.
$$
Here $p^i \ell_i \in \pi_*(BP)$ is
recursively determined by
$$
p \ell_n = \sum_{i=0}^{n-1} \ell_i v_{n-i}^{p^i}
= v_n + \ell_1 v_{n-1}^p + \dots + \ell_{n-1} v_1^{p^{n-1}} \,.
$$
\end{theorem}

In Section~\ref{sec:tctc} we evaluate the $d^2$-differential
in the circle Tate spectral sequences for $MU$ and $BP$,
in a finite range of degrees, and use this to calculate
the resulting $E^3 = E^4$-term.  For example
$$
E^4 = \bZ[t, t^{-1}] \otimes H(\pi_* THH(BP), \sigma)
	\Longrightarrow \pi_* THH(BP)^{tS^1}
$$
where
\begin{multline*}
H(\pi_* THH(BP), \sigma) \\
= \begin{cases}
\bZ_{(p)}\{1\} & \text{for $*=0$,} \\
\bZ/p\{v_1^{i-1} \lambda_1\}
        & \text{for $* = i(2p-2)+1$, $1 \le i \le p-1$,} \\
\bZ/p^2\{v_1^{p-1} \lambda_1\} & \text{for $* = 2p^2 - 2p + 1$,} \\
\bZ/p^2\{\lambda_2\} & \text{for $* = 2p^2 - 1$,} \\
\bZ_{(p)}/p^2(p+2) \{v_2 \lambda_1 + v_1 \lambda_2\}
        & \text{for $* = 2p^2 + 2p - 3$,} \\
\bZ/p\{\lambda_1 \lambda_2\} & \text{for $* = 2p^2 + 2p - 2$,} \\
0 & \text{for the remaining $* \le 2p^2 + 4p - 6$.}
\end{cases}
\end{multline*}
This appears as Theorem~\ref{thm:HTHHBPsigma} in the body of the
paper.  In particular, we see that while $H(\pi_* THH(BP), \sigma)$ is
concentrated in odd degrees for $0 < * < 2p^2 + 2p - 2$, this ceases to
be true in degree $|\lambda_1 \lambda_2| = 2p^2 + 2p - 2$.

The cyclic structure on $THH(R)$ suffices to define the circle
homotopy fixed points $THH(R)^{hS^1}$ and the circle Tate construction
$THH(R)^{tS^1}$.  When enriched to a cyclotomic structure \cite{BHM93},
\cite{NS18}, these data suffice to define the topological cyclic homology
$TC(R)$, which is a powerful invariant \cite{DGM13} of the algebraic
$K$-theory $K(R)$, especially for connective $S$-algebras~$R$.
The graded rings $\pi_* THH(\bF_p)$ and $\pi_* THH(\bZ)$ were
calculated by B{\"o}kstedt, and formed the basis for calculations of
$TC(\bF_p)$ and $TC(\bZ)$, see \cite{HM97} for the case of the prime
field~$\bF_p$, \cite{BM94} and \cite{BM95} for the integers localized
at an odd prime~$p$, and \cite{Rog98}, \cite{Rog99a}, \cite{Rog99b}
and \cite{Rog99c} for the integers localized at $p=2$.  The topological
Hochschild homology of $R = \ell$ (the Adams summand in
$p$-local connective complex $K$-theory) was worked out for $p\ge5$
in~\cite{MS93} and promoted to a calculation of $TC(\ell)$
in~\cite{AR02}.  In all of these cases, the $\sigma$-operator acts
trivially on $\pi_* THH(R)$.

When $R = S$ the circle action on $THH(R) = S$ is trivial, so the
$d^2$-differential in the circle Tate spectral sequence alternates
between zero and multiplication by $\eta$ in $\pi_* THH(R)
= \pi_*(S)$.  The resulting Tate spectral sequence agrees with the
Atiyah--Hirzebruch spectral sequence for $S^{tS^1} \simeq \Sigma^2
\bC P^\infty_{-\infty}$.  Knowledge of the attaching maps in complex
projective spaces translates to substantial knowledge \cite{Mos68}
of the differential patterns in this spectral sequence.  However, the
limits on our knowledge of $\pi_*(S)$ put bounds on how well we can
understand $TC(S)$ and $K(S)$ by this approach.  Explicit calculations
in low degrees were made in~\cite{Rog02}, \cite{Rog03} and~\cite{BM19},
but it would be desirable to place these in a context of systematic
patterns, similar to the chromatic filtration in stable homotopy theory
\cite{MRW77}, \cite{Rav92}.  By the descent results of \cite{DR18}, a
good understanding of $\pi_* TC(MU \wedge \dots \wedge MU)$ (with one or
more copies of~$MU$) will also determine $\pi_* TC(S)$, through a homotopy
limit or descent spectral sequence.  The problem of determining $TC(MU)$
and $K(MU)$ has therefore been frequently considered, e.g.~by Ausoni
and the author at the time when~\cite{AR02} was completed.  In this case
the $\sigma$-operator acts nontrivially on $\pi_* THH(MU)$, but precise
formulas seem not to have been worked out before this paper.

The author has also pursued a homological approach \cite{BR05} to the
calculations of $THH(R)^{hS^1}$ and $THH(R)^{tS^1}$ for $S$-algebras such
as $R = MU$, working with continuous homology in the category of completed
$A_*$-comodule algebras.  This led, in \cite{LNR12},~\cite{LNR11} and
\cite{BBLNR14} to a proof that there are $p$-adic equivalences
$$
THH(MU)^{hS^1} \overset{\Gamma}\longleftarrow
TF(MU; p) \overset{\hat\Gamma}\longto THH(MU)^{tS^1}
$$
where $TF(MU; p) = \holim_n THH(MU)^{C_{p^n}}$.  This provides the
foundation for a calculation of $\pi_* THH(MU)^{C_{p^n}}$ by induction
on~$n$.  We plan to discuss the homological approach to
$THH(MU)^{tS^1}$ in a future paper.

\section{Formal group laws and moving coordinates}
\label{sec:fglamc}

\subsection{Formal group laws and complex cobordism}

The universal (commutative, $1$-dimensional) formal group law
$$
F(x, y) = x + y + \sum_{i,j\ge1} a_{ij} x^i y^j
$$
is defined over the Lazard ring
$L = \bZ[a_{ij} \mid i,j\ge1] / I$
where
\begin{multline*}
I = (a_{12} - a_{21}, a_{13} - a_{31}, a_{14} - a_{41}, a_{23} - a_{32},
        \dots, \\
2 a_{11} a_{12} + 3 a_{13} - 2 a_{22},
2 a_{12}^2 + 3 a_{11} a_{13} + 4 a_{14} - 2 a_{23}, \\
a_{11}^2 a_{12} + 3 a_{12}^2 + 6 a_{11} a_{13} - a_{11} a_{22} + 6 a_{14} - 3 a_{23}, \dots)
\end{multline*}
is the ideal generated by the coefficients of $x^i y^j$ in $F(x, y)
- F(y, x)$ and of $x^i y^j z^k$ in $F(F(x, y), z) - F(x, F(y, z))$.
Each ring homomorphism $\theta \: L \to R$ determines a formal group law
$$
(\theta_* F)(x, y) = x + y + \sum_{i,j\ge1} \theta(a_{ij}) x^i y^j
$$
defined over~$R$, and this specifies a natural bijection between ring
homomorphisms $L \to R$ and formal group laws defined over $R$.  We say
that $L$ classifies, or corepresents, formal group laws.  We give $L$
the grading where $a_{ij}$ has degree $2(i+j-1)$.  Lazard~\cite{Laz55}
proved that $L \cong \bZ[x_n \mid n\ge1]$ where $x_n$ has degree~$2n$.
Following Adams~\cite{Ada74}*{p.~57}, augmented with a little computer
algebra, we can take
\begin{align*}
x_1 &= a_{11} \\
x_2 &= a_{12} \\
x_3 &= a_{22} - a_{13} \\
x_4 &= a_{14}
\end{align*}
as the first four generators of the Lazard ring.
Quillen~\cite{Qui69}*{Thm.~2} showed that the tensor product formula
for the first Chern class in complex cobordism theory specifies a formal
group law over $\pi_*(MU)$, and that the homomorphism $L \to \pi_*(MU)$
that classifies this formal group law is an isomorphism.

\subsection{Strict isomorphisms}
Let $F$ and $F'$ be formal group laws defined over~$R$.  A strict
isomorphism $f \: F \to F'$ over~$R$ is a formal power series
$$
f(x) = x + \sum_{n\ge1} b_n x^{n+1}
$$
with $b_n \in R$, such that $f(F(x,y)) = F'(f(x), f(y))$.  If $R$
is torsion-free then there is at most one strict isomorphism from $F$
to $F'$.  Let $B = \bZ[b_n \mid n\ge1]$, graded so that $b_n$ has
degree~$2n$.  The tensor product $LB = L \otimes B$ then classifies
diagrams
$$
F \overset{f}\longto F'
$$
where $F$ and $F'$ are formal group laws and $f$ is a strict isomorphism.
Restriction along the inclusions $\eta_L \: L \to LB$ and $\iota \:
B \to LB$ let us recover $F$ and~$f$, respectively, while restriction
along the right unit $\eta_R \: L \to LB$ classifies $F'$.
Continuing Adams' calculations \cite{Ada74}*{p.~63}, we have
\begin{align*}
\eta_R(a_{11}) &= a_{11} + 2 b_1 \\
\eta_R(a_{12}) &= a_{12} + a_{11} b_1 + (3 b_2 - 2 b_1^2) \\
\eta_R(a_{13}) &= a_{13} + a_{11} (2 b_2 - 2 b_1^2)
        + (4 b_3 - 8 b_1 b_2 + 4 b_1^3) \\
\eta_R(a_{22}) &= a_{22} + (2 a_{12} + a_{11}^2) b_1
        + a_{11} (6 b_2 - 3 b_1^2) + (6 b_3 - 6 b_1 b_2 + 2 b_1^3) \\
\eta_R(a_{14}) &= a_{14} - a_{13} b_1 + a_{12} (b_2 - b_1^2)
	+ a_{11} (3 b_3 - 8 b_1 b_2 + 5 b_1^3) \\
&\qquad + (5 b_4 - 14 b_1 b_3 - 6 b_2^2 + 25 b_1^2 b_2 - 10 b_1^4) \\
\eta_R(a_{23}) &= a_{23} + a_{13} b_1 + 2 a_{11} a_{12} b_1
+ a_{12} (8 b_2 - 6 b_1^2)
+ a_{11}^2 (3 b_2 - 2 b_1^2) \\
&\qquad + a_{11} (12 b_3 - 16 b_1 b_2 + 6 b_1^3)
+ (10 b_4 - 16 b_1 b_3 - 3 b_2^2 + 14 b_1^2 b_2 - 4 b_1^4)
\,,
\end{align*}
where we have corrected a (rare) typographical error in Adams' formula
for $\eta_R(a_{22})$.  It follows that
\begin{equation} \label{eq:etaRxn-xnbn}
\begin{aligned}
\eta_R(x_1) &= x_1 + 2 b_1 \\
\eta_R(x_2) &= x_2 + x_1 b_1 + (3 b_2 - 2 b_1^2) \\
\eta_R(x_3) &= x_3 + (2 x_2 + x_1^2) b_1
        + x_1 (4 b_2 - b_1^2) + (2 b_3 + 2 b_1 b_2 - 2 b_1^3) \\
\eta_R(x_4) &= x_4 + (2 x_1 x_2 - 2 x_3) b_1 + x_2 (b_2 - b_1^2)
        + x_1 (3 b_3 - 8 b_1 b_2 + 5 b_1^3) \\
&\qquad + (5 b_4 - 14 b_1 b_3 - 6 b_2^2 + 25 b_1^2 b_2 - 10 b_1^4) \,.
\end{aligned}
\end{equation}

\subsection{Hopf algebroids}
The formal group laws and strict isomorphisms defined over~$R$ form the
objects and morphisms of a small groupoid $\cG(R)$, depending functorially
on the commutative ring~$R$.  The identity morphism $id_F \: F \to F$
is the formal power series $id_F(x) = x$, which is classified by an
augmentation $\epsilon \: LB \to L$.  It has the form $id \otimes
\epsilon$, where $\epsilon \: B \to \bZ$ maps each $b_n$ to zero.
The composite of two strict isomorphisms $f \: F \to F'$ and $f' \: F'
\to F''$ is a strict isomorphism $f' f \: F \to F''$.  The composition
pairing $\circ$ is a natural function
\begin{align*}
\cG(R)(F', F'') \times \cG(R)(F, F') &\longto \cG(R)(F, F'') \\
\circ \: (f', f) &\longto f' f \,.
\end{align*}
The opposite pairing $\bullet \: (f, f') \mapsto f' f$ is classified
by a coproduct $\psi \: LB \to LB \otimes_L LB$, where $L$ acts through
$\eta_R$ on the left hand tensor factor and through $\eta_L$ on the right
hand tensor factor.  When composed with the obvious isomorphism $LB \otimes_L
LB \cong L \otimes B \otimes B$ it takes the form $id \otimes \psi$,
where $\psi \: B \to B \otimes B$ defines the coproduct in a Hopf algebra
structure on~$B$.  In low degrees,
\begin{align*}
\psi(b_1) &= b_1 \otimes 1 + 1 \otimes b_1 \\
\psi(b_2) &= b_2 \otimes 1 + 2 b_1 \otimes b_1 + 1 \otimes b_2 \\
\psi(b_3) &= b_3 \otimes 1 + (b_1^2 + 2 b_2) \otimes b_1 + 3 b_1 \otimes b_2
        + 1 \otimes b_3 \\
\psi(b_4) &= b_4 \otimes 1 + (2 b_1 b_2 + 2 b_3) \otimes b_1
	+ (3 b_1^2 + 3 b_2) \otimes b_2 + 4 b_1 \otimes b_3 + 1 \otimes b_4 \,,
\end{align*}
see \cite{Ada74}*{p.~91}.
The inverse $f^{-1} \: F' \to F$ of a strict isomorphism $f \: F \to
F'$ is classified by a homomorphism $\chi \: LB \to LB$.  Its restriction along
$\eta_L \: L \to LB$ is $\eta_R$, while its restriction along $\iota \:
B \to LB$ is $\iota \chi$, where $\chi \: B \to B$ is the conjugation
in the Hopf algebra structure on~$B$.
Following \cite{Ada74}*{p.~65}, for $f(x) = x + \sum_{n\ge1} b_n x^{n+1}$
we have
$$
f^{-1}(x) = x + \sum_{n\ge1} \bar b_n x^{n+1}
$$
where $\bar b_n = \chi(b_n)$ is given in low degrees by
\begin{align*}
\bar b_1 &= - b_1 \\
\bar b_2 &= 2 b_1^2 - b_2 \\
\bar b_3 &= - 5 b_1^3 + 5 b_1 b_2 - b_3 \\
\bar b_4 &= 14 b_1^4 - 21 b_1^2 b_2 + 3 b_2^2 + 6 b_1 b_3 - b_4 \,.
\end{align*}

One might now like to say that $L$, $LB$, $\eta_L$, $\eta_R$ and $\psi$
classify the objects, morphisms, sources, targets and composition in
the groupoid $\cG(R)$, but due to the reversal of ordering in the pairing
$\bullet$, this is not quite correct.  Instead, $L$ and $LB$ classify
the objects and morphisms in the opposite groupoid, $\cG^{op}(R)$.
A homomorphism $LB \to R$ corresponds to a diagram $f \: F \to F'$, as
above, which we can view as a morphism in $\cG^{op}(R)$ with source $F'$ and
target $F$.  Then $\eta_L$ and $\eta_R$ classify the target and source,
respectively, and $\psi$ classifies the composition $(f, f') \mapsto f
\bullet f'$ in~$\cG^{op}(R)$, where $f' \: F' \to F''$ is as before.

Alternatively, we can focus on the inverse strict isomorphism
$\phi = f^{-1} \: F' \to F$, in place of $f \: F \to F'$.  
We think of $L$ as classifying formal group laws in the same
way as before, but now we think of $LB$ as classifying diagrams
$$
F \overset{\phi}\longfrom F'
$$
where $F$ and $F'$ are formal group laws and $\phi$ is the strict
isomorphism
$$
\phi(x) = x + \sum_{n\ge1} \bar b_n x^{n+1}
$$
with $\bar b_n \in B$ included by $\iota \: B \to LB$.
The $\bar b_n$ provide a second polynomial basis for $LB$ over
$L$, so that $LB \cong L\bar B = L[\bar b_n \mid n\ge1]$.
Then $L$ and $LB$ corepresent the objects and morphisms in $\cG(R)$,
$\eta_L$ and $\eta_R$ corepresent the target and source of a morphism,
in that order, $\epsilon$ corepresents the identity morphism,
$\psi$ corepresents the composition
pairing
\begin{align*}
\cG(R)(F', F) \times \cG(R)(F'', F') &\longto \cG(R)(F'', F) \\
\circ \: (\phi, \phi') &\longto \phi \phi' \,,
\end{align*}
and $\chi$ corepresents passage to the inverse of a morphism.

Novikov~\cite{Nov67} and Landweber~\cite{Lan67} studied the cohomology
operations in complex cobordism, which are represented by classes in
$MU^*(MU)$.  Turning instead to homology, Adams
\cite{Ada74}*{Lem.~4.5(ii)} showed that
$$
\pi_*(MU \wedge MU) \cong \pi_*(MU) [b_n \mid n\ge1]
$$
for specific classes $b_n \in \pi_*(MU \wedge MU)$, so that Quillen's
isomorphism $L \cong \pi_*(MU)$ extends to an isomorphism $LB \cong
\pi_*(MU \wedge MU)$.  By the results of \cite{Ada74}*{\S11}, the left and
right units $\eta_L, \eta_R \: L \to LB$ correspond to the homomorphisms
induced by the maps $MU \cong MU \wedge S \to MU \wedge MU$ and $MU \cong
S \wedge MU \to MU \wedge MU$, respectively.  Likewise, the augmentation
$\epsilon \: LB \to L$ is induced by the multiplication $MU \wedge
MU \to MU$.  The coproduct $\psi \: LB \to LB \otimes_L LB$
is induced by the map $MU \wedge MU \cong MU \wedge S \wedge MU \to MU
\wedge MU \wedge MU$, via the isomorphism
$$
\pi_*(MU \wedge MU) \otimes_{\pi_*(MU)} \pi_*(MU \wedge MU)
	\overset{\cong}\longto \pi_*(MU \wedge MU \wedge MU) \,.
$$
Finally, the conjugation $\chi \: LB \to LB$ is induced by the twist map
$\tau \: MU \wedge MU \cong MU \wedge MU$.  In all cases the unlabeled
map $S \to MU$ is the unit map in the ring spectrum structure.  Using the
terminology introduced by Haynes Miller \cite{Mil75}, $(L, LB)$ and
$(\pi_*(MU), \pi_*(MU \wedge MU))$ are isomorphic as Hopf algebroids.

\subsection{Moving coordinates}
So far we have classified strict isomorphisms $f(x) = x + \sum_{n\ge1}
b_n x^{n+1}$ or $\phi(x) = x + \sum_{n\ge1} \bar b_n x^{n+1}$ in a
way that is independent of the source and target of $f$ and~$\phi$.
Such ``absolute coordinates'' exist, because the Hopf algebroid $(L,
LB)$ is split.  Following Araki \cite{Ara73}*{Prop.~2.10} and Landweber
\cite{Lan75}, we can instead classify strict isomorphisms in terms of
``moving coordinates''.  This will lead to nicer formulas for the
$\sigma$-operator in $\pi_* THH(MU)$.  The strict isomorphism
$$
F \overset{\phi}\longfrom F'
$$
can be uniquely written as a formal sum
$$
\phi(x) = x +_F {\sum_{n\ge1}}^F c_n x^{n+1} \,,
$$
with respect to the target formal group law, for a sequence of elements
$c_n \in LB$ with $c_n$ in degree~$2n$.  (In spite of the notation,
these are essentially unrelated to the Chern classes.)  In low degrees,
\begin{align*}
c_1 &= \bar b_1 \\
c_2 &= - a_{11} \bar b_1 + \bar b_2 \\
c_3 &= - a_{12} \bar b_1 + a_{11}^2 \bar b_1 - a_{11} \bar b_2 + \bar b_3 \\
c_4 &= (a_{11}^2 - a_{12}) \bar b_1^2 - (a_{11}^3 - 2 a_{11}
a_{12} + a_{13}) \bar b_1 + (a_{11}^2 - a_{11} \bar b_1 - a_{12}) \bar
b_2 - a_{11} \bar b_3 + \bar b_4
\end{align*}
so that
\begin{align*}
c_1 &= -b_1 \\
c_2 &= a_{11} b_1 + (2 b_1^2 - b_2) \\
c_3 &= a_{12} b_1 - a_{11}^2 b_1 + a_{11} (b_2 - 2 b_1^2)
	+ (-5 b_1^3 + 5 b_1 b_2 - b_3) \\
c_4 &= 14 b_1^4 + (a_{11}^2 - a_{12}) b_1^2 - 21 b_1^2 b_2
	+ (a_{11}^2 + a_{11} b_1 - a_{12}) (2 b_1^2 - b_2) \\
	&\qquad
	+ a_{11} (5 b_1^3 - 5 b_1 b_2 + b_3)
	+ (a_{11}^3 - 2 a_{11} a_{12} + a_{13}) b_1
	+ 3 b_2^2 + 6 b_1 b_3 - b_4 \,.
\end{align*}
Hence
\begin{equation} \label{eq:cn-xnbn}
\begin{aligned}
c_1 &= -b_1 \\
c_2 &= x_1 b_1 + (2 b_1^2 - b_2) \\
c_3 &= x_2 b_1 - x_1^2 b_1 + x_1 (b_2 - 2 b_1^2)
	+ (-5 b_1^3 + 5 b_1 b_2 - b_3) \\
c_4 &= 14 b_1^4 + (x_1^2 - x_2) b_1^2 - 21 b_1^2 b_2
	+ (x_1^2 + x_1 b_1 - x_2) (2 b_1^2 - b_2) \\
	&\qquad + x_1 (5 b_1^3 - 5 b_1 b_2 + b_3)
	+ (x_1^3 - 4 x_1 x_2 + 2 x_3) b_1
	+ 3 b_2^2 + 6 b_1 b_3 - b_4 \,.
\end{aligned}
\end{equation}

The moving coordinates $c_n$ form yet another polynomial basis for $LB$
over $L$, so that $LB \cong LC = L[c_n \mid n\ge1]$.  This specifies
an isomorphism of Hopf algebroids $(L, LB) \cong (L, LC)$.  The left
unit $\eta_L \: L \to LC$ is given by the evident inclusion, and the
augmentation $\epsilon \: LC \to L$ sends each $c_n$ to zero, for
$n\ge1$.  The right unit $\eta_R \: L \to LC$ corepresents the source
$F'$ of the strict isomorphism $\phi \: F' \to F$ defined as above.
In the next subsection we shall obtain a useful formula for this right
unit homomorphism.

\subsection{Logarithms}
The additive formal group law $F_a$ is defined by $F_a(x, y) = x + y$.
Working over $L \otimes \bQ \cong \pi_*(MU) \otimes \bQ$ there is a
unique strict isomorphism
$$
F_a \overset{\log}\longfrom F
$$
from the universal formal group law to the additive one, which we can write
as
$$
\log(x) = x + \sum_{n\ge1} m_n x^{n+1}
$$
for unique elements $m_n \in L \otimes \bQ$, with $m_n$ in degree~$2n$.
See \cite{Ada74}*{Cor.~7.15} or \cite{Rav86}*{Thm.~A2.1.6}.
Let
$\exp(x) = x + \sum_{n\ge1} \bar m_n x^{n+1}$
be the inverse strict isomorphism, from $F_a$ to $F$.  Then
$\log F(x, y) = \log(x) + \log(y)$ and
$$
F(x,y) = \exp(\log(x) + \log(y))
$$
over $L \otimes \bQ$.  We can express the $\bar m_n$, the
$a_{ij}$ and the $x_n$ as integer polynomials in the logarithmic
coefficients~$m_n$.  In low degrees,
\begin{align*}
\bar m_1 &= - m_1 \\
\bar m_2 &= 2 m_1^2 - m_2 \\
\bar m_3 &= - 5 m_1^3 + 5 m_1 m_2 - m_3 \\
\bar m_4 &= 14 m_1^4 - 21 m_1^2 m_2 + 3 m_2^2 + 6 m_1 m_3 - m_4
\end{align*}
and
\begin{align*}
a_{11} &= - 2 m_1 \\
a_{12} &= 4 m_1^2 - 3 m_2 \\
a_{13} &= - 8 m_1^3 + 12 m_1 m_2 - 4 m_3 \\
a_{22} &= - 20 m_1^3 + 24 m_1 m_2 - 6 m_3 \\
a_{14} &= 16 m_1^4 - 36 m_1^2 m_2 + 9 m_2^2 + 16 m_1 m_3 - 5 m_4 \\
a_{23} &= 72 m_1^4 - 132 m_1^2 m_2 + 27 m_2^2 + 44 m_1 m_3 - 10 m_4 \,.
\end{align*}
Hence
\begin{equation} \label{eq:xn-mn}
\begin{aligned}
x_1 &= - 2 m_1 \\
x_2 &= 4 m_1^2 - 3 m_2 \\
x_3 &= - 12 m_1^3 + 12 m_1 m_2 - 2 m_3 \\
x_4 &= 16 m_1^4 - 36 m_1^2 m_2 + 9 m_2^2 + 16 m_1 m_3 - 5 m_4 \,.
\end{aligned}
\end{equation}

The resulting homomorphism $L \to \bZ[m_n \mid n\ge1]$
then induces an isomorphism $L \otimes \bQ \cong \bQ[m_n \mid n\ge1]$, so
after rationalization the classes $m_n$ serve as another set of polynomial
generators for $L \cong \bZ[x_n \mid n\ge1]$.  The rational classes $m_n$
are canonically defined, as opposed to the integral classes $x_n$.

The right unit $\eta_R \: L \to LB$ can be calculated using
the identity
\begin{equation} \label{eq:etaRmn-in-LB}
\sum_{n\ge0} \eta_R(m_n)
= \sum_{i\ge0} m_i \Bigl( \sum_{j\ge0} \bar b_j \Bigr)^{i+1}
\end{equation}
in $LB \otimes \bQ$, where $m_0 = 1$, $\bar b_0 = 1$ and $\eta_R(m_n)$
is equal to the degree~$2n$ part of either side of the formula.  See
\cite{Ada74}*{Prop.~9.4} or \cite{Rav86}*{Thm.~A2.1.16}.
Working instead with moving coordinates we obtain the following
formula, which does not seem to appear in the standard references.

\begin{proposition} \label{prop:etaR-on-mn}
The right unit $\eta_R \: L \to LC$ is determined by the formula
$$
\eta_R(m_n) = \sum_{(i+1)(j+1) = n+1} m_i c_j^{i+1}
$$
in $LC \otimes \bQ$, where $m_0 = 1$ and $c_0 = 1$.
The sum runs over the indices $i,j\ge0$ with $(i+1)(j+1) = n+1$.
\end{proposition}

\begin{proof}
The proofs of \cite{Ada74}*{Thm.~16.1(i)} and \cite{Lan75}*{Thm.~3(i)}
readily carry over from the $p$-typical situation to the general one.
The formal sum
$$
\phi(x) = {\sum_{n\ge0}}^F c_n x^{n+1}
$$
defines a strict isomorphism
$$
(\eta_L)_*(F) \overset{\phi}\longfrom (\eta_R)_*(F)
$$
of formal group laws over $LC$.
The strict isomorphism $\log \: F \to F_a$ over $L \otimes \bQ$
induces strict isomorphisms
\begin{align*}
(\eta_L)_*(\log) \: (\eta_L)_*(F) &\longto (\eta_L)_*(F_a) = F_a \\
(\eta_R)_*(\log) \: (\eta_R)_*(F) &\longto (\eta_R)_*(F_a) = F_a
\end{align*}
over $LC \otimes \bQ$. By their uniqueness we must have
$$
(\eta_R)_*(\log) = (\eta_L)_*(\log) \circ \phi \,.
$$
Hence
$$
\sum_{n\ge0} \eta_R(m_n) x^{n+1}
= \log\Bigl({\sum_{j\ge0}}^F c_j x^{j+1}\Bigr)
= \sum_{j\ge0} \log(c_j x^{j+1})
= \sum_{i,j\ge0} m_i (c_j x^{j+1})^{i+1} \,.
$$
Concentrating on the coefficients of $x^{n+1}$ yields the formula.
\end{proof}

\section{The $p$-typical case}
\label{sec:typ}

Let $p$ be any prime.  A ($1$-dimensional, commutative) formal group
law $F$ over a torsion-free $\bZ_{(p)}$-algebra $R$ is $p$-typical
\cite{Car67a}, \cite{Car67b}, \cite{Rav86}*{Def.~A2.1.17} if its
logarithmic coefficients satisfy $m_n = 0$ unless $n+1$ is a power of~$p$.
In other words,
$$
\log(x) = x + \sum_{n\ge1} \ell_n x^{p^n}
$$
for a sequence of coefficients $\ell_n \in R \otimes \bQ$, with $\ell_n$
in degree~$2(p^n-1)$.  There is a universal $p$-typical formal group
law~$F$, defined over the $\bZ_{(p)}$-subalgebra $V \subset
\bQ[\ell_n \mid n\ge1]$ generated by the coefficients of
the formal power series
$$
F(x,y) = \log^{-1}(\log(x) + \log(y)) \,.
$$
By the universal property of the Lazard ring, there is a ring
homomorphism $\alpha \: L \otimes \bZ_{(p)} \to V$ classifying the
underlying formal group law of~$F$.  Conversely, each formal group
law over a $\bZ_{(p)}$-algebra is strictly isomorphic to a unique
$p$-typical one.  Hence there is a ring homomorphism $\beta \: V \to
L \otimes \bZ_{(p)}$ classifying the $p$-typification of the universal
formal group law.  The composite $\alpha\beta \: V \to L \otimes \bZ_{(p)}
\to V$ is the identity, and the composite $e = \beta\alpha \: L \otimes
\bZ_{(p)} \to V \to L \otimes \bZ_{(p)}$ is (the Quillen) idempotent.
After rationalization, $e(m_n) = m_n$ if $n+1$ is a power of $p$, and
$e(m_n) = 0$ otherwise.  It follows that
$$
V = \bZ_{(p)}[v_n \mid n\ge1] \,,
$$
with $v_n$ in degree~$2(p^n-1)$, and $V \otimes \bQ \to \bQ[\ell_n \mid
n\ge1]$ is an isomorphism.
One choice of generators $v_n$, due to Hazewinkel \cite{Haz76}*{(4.3.1)},
is recursively defined by
\begin{equation} \label{eq:Hazewinkel}
p \ell_n = \sum_{i=0}^{n-1} \ell_i v_{n-i}^{p^i}
= v_n + \ell_1 v_{n-1}^p + \dots + \ell_{n-1} v_1^{p^{n-1}} \,.
\end{equation}
Here $\ell_0 = 1$.  In low degrees,
\begin{equation} \label{eq:pnelln}
\begin{aligned}
p \ell_1 &= v_1 \\
p^2 \ell_2 &= p v_2 + v_1^{p+1} \\
p^3 \ell_3 &= p^2 v_3 + p (v_1 v_2^p + v_1^{p^2} v_2) + v_1^{p^2+p+1} \,.
\end{aligned}
\end{equation}
These formulas exhibit one advantage of the $p$-typical context over
the general one.  We do not have similar formulas characterizing
polynomial generators $x_n \in L$ in terms of the logarithmic
coefficients $m_n \in L \otimes \bQ$.

\begin{lemma} \label{lem:pnellninV}
$p^n \ell_n \in V$ for each $n\ge1$.
\end{lemma}

\begin{proof}
This follows from \eqref{eq:Hazewinkel} by induction on~$n$.
\end{proof}

Quillen \cite{Qui69}*{Thm.~4} constructed the Brown--Peterson spectrum $BP$
with maps $\alpha \: MU_{(p)} \to BP$ and $\beta \: BP \to MU_{(p)}$,
such that $\alpha \beta$ is homotopic to the identity and $\beta \alpha$
induces the idempotent $e$ on homotopy.  The resulting formal group
law over $\pi_*(BP)$ is then $p$-typical, and the homomorphism $V
\to \pi_*(BP)$ that classifies this $p$-typical formal group law is
an isomorphism.

Consider a strict isomorphism
$$
F \overset{\phi}\longfrom F'
$$
given in moving coordinates by
$$
\phi(x) = {\sum_{n\ge0}}^F c_n x^{n+1} \,.
$$
By Araki \cite{Ara73}*{Thm.~3.6} and Landweber \cite{Lan75}*{Lem.~1}, $F'$
is $p$-typical if and only if $c_n = 0$ unless $n+1$ is a power of~$p$.
In other words,
$$
\phi(x) = {\sum_{n\ge0}}^F t_n x^{p^n}
$$
for a sequence of coefficients $t_n$, with $t_0 = 1$.  Let
$$
T = \bZ_{(p)}[t_n \mid n\ge1]
$$
and $VT = V \otimes T$,
with $t_n$ in degree~$2(p^n-1)$.  Then $V$ and $VT$ corepresent the
objects and morphism in the full subgroupoid $\cT(R) \subset \cG(R)$
of $p$-typical formal group laws and their strict isomorphisms.
The inclusion $\eta_L \: V \to VT$ and a right unit homomorphism
$\eta_R \: V \to VT$ classify the target~$F$ and the
source~$F'$.  The augmentation $\epsilon \: VT \to V$ mapping each
$t_n$ to zero classifies the identity morphism.  A coproduct
$\psi \: VT \to VT \otimes_V VT$ classifies the composition
$$
\circ \: \cT(R)(F', F) \times \cT(R)(F'', F') \longto
\cT(R)(F'', F) \,,
$$
and a conjugation $\chi \: VT \to VT$ classifies the function sending
$\phi$ to $\phi^{-1} \: F \to F'$.  The pair $(V, VT)$, equipped with
these structure maps, is then a Hopf algebroid, corepresenting $\cT(R)$
as a functor of commutative $\bZ_{(p)}$-algebras.  The full inclusion
$\cT(R) \subset \cG(R)$ and the $p$-typification functor $\cG(R) \to
\cT(R)$ are classified by morphisms $\alpha \: (L, LC) \otimes \bZ_{(p)}
\to (V, VT)$ and $\beta \: (V, VT) \to (L, LC) \otimes \bZ_{(p)}$
of Hopf algebroids.  Here $\alpha \: L \otimes \bZ_{(p)} \to V$ is
given rationally by $\alpha(m_n) = 0$ if $n+1$ is not a power of~$p$
and $\alpha(m_n) = \ell_i$ if $n+1 = p^i$.  Similarly, $\alpha(c_n) = 0$
if $n+1$ is not a power of~$p$ and $\alpha(c_n) = t_i$ if $n+1 = p^i$.
Conversely, $\beta(\ell_i) = m_{p^i-1}$ and $\beta(t_i) = c_{p^i-1}$.

Adams \cite{Ada74}*{Thm.~16.1(ii)} showed that
$$
\pi_*(BP \wedge BP) \cong \pi_*(BP) [t_n \mid n\ge1]
$$
for specific classes $t_n \in \pi_*(BP \wedge BP)$, so that Quillen's
isomorphism $V \cong \pi_*(BP)$ extends to an isomorphism $VT \cong
\pi_*(BP \wedge BP)$.  Adams showed that the formulas for the flat
Hopf algebroid $(\pi_*(BP), \pi_*(BP \wedge BP))$, associated to the
(homotopy commutative or better) ring spectrum $BP$, agree with those
specified by Quillen.  Landweber \cite{Lan75}*{\S3} then verified that
these agree with the structure maps of $(V, VT)$, corepresenting $\cT(R)$.

In particular, \cite{Qui69}*{Thm.~5(iii)} and \cite{Ada74}*{Thm.~16.1(i)}
gave a formula for the right unit~$\eta_R \: V \to VT$ after
rationalization, which in our notation reads
\begin{equation} \label{eq:etaRelln}
\eta_R(\ell_n) = \sum_{i+j=n} \ell_i t_j^{p^i}
	= \ell_n + \ell_{n-1} t_1^{p^{n-1}}
	+ \dots + \ell_1 t_{n-1}^p + t_n
\end{equation}
for $n\ge1$.
When combined with~\eqref{eq:Hazewinkel}, this will give us
good formulas for the $\sigma$-operator in $\pi_* THH(BP)$.

\section{Topological Hochschild homology and the bar construction}
\label{sec:thhatbc}

\subsection{Chains of composable strict isomorphisms}
Recall that $THH(MU)_\bullet$ is a simplicial ($E_\infty$ ring)
spectrum $[q] \mapsto THH(MU)_q$.  We shall analyze the simplicial
graded commutative ring $[q] \mapsto \pi_* THH(MU)_q$ in terms of the
Hopf algebroids $(L, LB)$ and $(L, LC)$.  We emphasize the latter case,
since it is closer to the Hopf algebroid $(V, VT)$ that we need to
consider in the $p$-typical case.  However, we shall also state some of
the results for $(L, LB)$, in part to illustrate the advantage of using
moving coordinates for these calculations.

The product map
$$
L \otimes L^{\otimes q} = \pi_*(MU) \otimes \pi_*(MU)^{\otimes q}
	\longto \pi_*(MU \wedge MU^{\wedge q}) = \pi_* THH(MU)_q
$$
is not an isomorphism (for $q\ge1$), but becomes one after
rationalization.  Since $LB \cong LC$ is flat over $L$, we can rewrite its
target as
$$
\pi_*((MU \wedge MU) \wedge_{MU} \dots \wedge_{MU} (MU \wedge MU))
\cong LC \otimes_L \dots \otimes_L LC \,,
$$
with $q$ copies of $MU \wedge MU$ and $LC$.
Here
$LC \otimes_L \dots \otimes_L LC \cong L \otimes C^{\otimes q}$
by iterated use of standard isomorphisms of the form $X \otimes_L L
\otimes Y \cong X \otimes Y$, for suitable~$X$ and~$Y$.  The tensor
product $\pi_* THH(MU)_q \cong L \otimes C^{\otimes q}$ classifies chains
$$
(F_0; \phi_1, \dots, \phi_q) \quad = \quad
(F_0 \overset{\phi_1}\longfrom F_1 \longfrom \dots \overset{\phi_q}\longfrom F_q)
$$
of $q$ composable strict isomorphisms between formal group laws,
with $L$ classifying the formal group law $F_0$ and the $i$-th copy
of $C$ classifying the strict isomorphism $\phi_i \: F_i \to F_{i-1}$,
presented in moving coordinates with respect to its target.  The composite
homomorphism $L \otimes L^{\otimes q} \to L \otimes C^{\otimes q}$
classifies the function taking $(F_0; \phi_1, \dots, \phi_q)$ to the
$(1+q)$-tuple of formal group laws $(F_0, F_1, \dots, F_q)$.

The $i$-th face map $d_i \: \pi_* THH(MU)_{q+1} \to THH(MU)_q$ is a
homomorphism $L \otimes C^{\otimes q+1} \to L \otimes C^{\otimes q}$.
When $0 \le i \le q$ it is induced by the multiplication $MU \wedge
MU \to MU$ of the $i$-th and $(i+1)$-th copies of $MU$ (counting from
zero), and is compatible with the multiplication $L \otimes L \to L$
of the $i$-th and $(i+1)$-th copies of $L$ in $L \otimes L^{\otimes q+1}$.
Hence it corepresents the function taking $(F_0; \phi_1, \dots, \phi_q)$
to a chain where $F_i$ has been duplicated in the $i$-th and $(i+1)$-th
positions.  Since the only strict automorphism of a formal group law
over a torsion-free ring is the identity, the chain of composable strict
isomorphisms must be of the form
$$
(F_0; \phi_1, \dots, \phi_i, id, \phi_{i+1}, \dots, \phi_q)
$$
where $id \: F_i \to F_i$ denotes the identity isomorphism of the
repeated formal group law.
The last face map, with $i = q+1$, is induced by multiplying the
final and initial copies of $MU$.  It is compatible with the
multiplication of the final and initial copies of $L$, and therefore
corepresents the function taking $(F_0; \phi_1, \dots, \phi_q)$ to a
chain where a copy of $F_0$ has been added at the end.  This chain must
have the form
$$
(F_0; \phi_1, \dots, \phi_q, (\phi_1 \cdots \phi_q)^{-1})
$$
where $(\phi_1 \cdots \phi_q)^{-1} \: F_0 \to F_q$.

The $j$-th degeneracy map $s_j \: \pi_* THH(MU)_{q-1} \to THH(MU)_q$ is
a homomorphism $L \otimes C^{\otimes q-1} \to L \otimes C^{\otimes q}$.
It is induced by inserting the unit map $S \to MU$ between the $j$-th and
$(j+1)$-th copies of $MU$ (counting from zero), and is compatible with
the unit map $\bZ \to L$ to the $(j+1)$-th copy of~$L$ in $L \otimes
L^{\otimes q}$.  It therefore corepresents the function taking $(F_0;
\phi_1, \dots, \phi_q)$ to a chain where $F_{j+1}$ has been omitted.
For $0 \le j \le q-2$ this is the chain $(F_0; \phi_1, \dots, \phi_{j+1}
\phi_{j+2}, \dots, \phi_q)$, while for $j=q-1$ we obtain $(F_0; \phi_1,
\dots, \phi_{q-1})$.

\subsection{Strict isomorphisms with common target}
Let $\beta(C)_\bullet$ be the simplicial bar construction
$$
[q] \mapsto \beta(C)_q = C^{\otimes q}
$$
on the augmented algebra~$C$.  The homology of its normalized chain
complex $N\beta(C)_*$ calculates
$$
\Tor^C_*(\bZ, \bZ) = E([c_n] \mid n\ge1) \,,
$$
since $C$ is flat over~$\bZ$.  Here $[c_n]$ is the class in $\Tor^C_1(\bZ,
\bZ) = I(C)/ I(C)^2$ of the bar $1$-cycle $c_n \in I(C)$,
where $I(C) = \ker(C \to \bZ) \cong \cok(\bZ \to C)$ is the
positive-degree part of $C$, and $E(-)$ denotes the exterior algebra on
the indicated generators.

Then $L \otimes \beta(C)_\bullet$
is a simplicial graded commutative ring with
$$
[q] \mapsto L \otimes \beta(C)_q = L \otimes C^{\otimes q} \,,
$$
and the homology of $L \otimes N\beta(C)_*$ calculates $L \otimes
\Tor^C_*(\bZ, \bZ)$.  We think of $L \otimes \beta(C)_q = L \otimes
C^{\otimes q}$ as classifying a $q$-tuple
$$
(F_0; \gamma_1, \dots, \gamma_q) \quad = \quad
(F_0 \overset{\gamma_i}\longfrom F_i)_{i=1}^q
$$
of strict isomorphisms $\gamma_i \: F_i \to F_0$ in moving coordinates,
all with the same target.

The $i$-th face map $d_i \: L \otimes C^{\otimes q+1} \to L \otimes
C^{\otimes q}$ is given for $i=0$ by the augmentation $C \to \bZ$ of the
first copy of $C$, for $1 \le i \le q$ by the multiplication $C \otimes
C \to C$ of the $i$-th and $(i+1)$-th copy of $C$ (counting from one),
and for $i = q+1$ by the augmentation of the last copy of~$C$.  This
corepresents the function that takes $(F_0; \gamma_1, \dots, \gamma_q)$
to $(F_0; id, \gamma_1, \dots, \gamma_q)$ for $i=0$, to $(F_0; \gamma_1,
\dots, \gamma_i, \gamma_i, \dots, \gamma_q)$ for $1 \le i \le q$, and to
$(F_0; \gamma_1, \dots, \gamma_q, id)$ for $i=q+1$.  In the first and last
cases, $id \: F_0 \to F_0$ refers to the identity isomorphism for $F_0$.

The $j$-th degeneracy map $s_j \: L \otimes C^{\otimes q-1} \to L \otimes
C^{\otimes q}$, for $0 \le j \le q-1$, is induced by the unit $\bZ \to C$
to the $(j+1)$-th copy of $C$ (counting from one).  It corepresents the
function that omits $\gamma_{j+1} \: F_{j+1} \to F_0$ from the $q$-tuple
of strict isomorphisms with target $F_0$, leaving $(F_0; \gamma_1, \dots,
\gamma_j, \gamma_{j+2}, \dots, \gamma_q)$.

\subsection{A shearing isomorphism}
To each chain of $q$ composable strict isomorphisms
$$
F_0 \overset{\phi_1}\longfrom F_1 \overset{\phi_2}\longfrom \dots
	\overset{\phi_q}\longfrom F_q
$$
we can associate a $q$-tuple of strict isomorphisms
$$
(F_0 \overset{\gamma_i}\longfrom F_i)_{i=1}^q \,,
$$
having the same underlying sequence $F_0, F_1, \dots, F_q$ of formal
group laws.  This one-to-one correspondence is classified by the following
isomorphism.

\begin{proposition}
There is an isomorphism of simplicial graded commutative rings
$$
L \otimes \beta(C)_\bullet
	\overset{\cong}\longto \pi_* THH(MU)_\bullet
$$
that, in degree~$q$, classifies the bijection
$$
(F_0; \phi_1, \dots, \phi_q) \overset{\cong}\longmapsto (F_0; \gamma_1, \dots, \gamma_q)
$$
where $\gamma_i = \phi_1 \cdots \phi_i$ for $1 \le i \le q$.  Its inverse is given
by $\phi_1 = \gamma_1$ and $\phi_i = \gamma_{i-1}^{-1} \gamma_i$ for $2
\le i \le q$.
\end{proposition}

\begin{proof}
The stated bijection is natural (in the ring over which the formal group
laws and the strict isomorphisms are defined), hence is corepresented
by an isomorphism
$$
L \otimes \beta(C)_q = L \otimes C^{\otimes q}
\overset{\cong}\longto
L \otimes C^{\otimes q} \cong \pi_* THH(MU)_q \,.
$$
This is the identity for $q=0$ and $q=1$, but has a more
complex expression for $q\ge2$, which we do not need to make explicit.

It remains to verify that these isomorphisms $L \otimes \beta(C)_q \cong
\pi_* THH(MU)_q$ are compatible with the simplicial structure maps.
This follows from the explicit descriptions given in the previous two
subsections: On both sides of the isomorphism the $i$-th face map,
except for the last one, classifies the function that repeats $F_i$
in the underlying sequence of formal group laws.  Similarly, on both
sides the last face map classifies the function that appends $F_0$ to
the underlying sequence.  This ensures that all face maps are compatible
under these isomorphisms.  Finally, on both sides the $j$-th degeneracy
map classifies the function that omits $F_{j+1}$ from the underlying
sequence.  This ensures that all degeneracy maps are compatible.
\end{proof}

Turning to the split case, we have an isomorphism
$$
\pi_* THH(MU)_q \cong LB \otimes_L \dots \otimes_L LB
	\cong L \otimes B^{\otimes q}
$$
and $L \otimes B^{\otimes q}$ classifies chains
$$
(F_0; f_1, \dots, f_q) \quad = \quad
(F_0 \overset{f_1}\longto F_1 \longto \dots
	\overset{f_q}\longto F_q)
$$
of $q$ composable strict isomorphisms.  On the other hand,
$L \otimes \beta(B)_q = L \otimes B^{\otimes q}$
classifies $q$-tuples
$$
(F_0, g_1, \dots, g_q) \quad = \quad
(F_0 \overset{g_i}\longto F_i)_{i=1}^q
$$
of strict isomorphisms, all with the same source.

\begin{proposition}
There is an isomorphism of simplicial graded commutative rings
$$
L \otimes \beta(B)_\bullet
	\overset{\cong}\longto \pi_* THH(MU)_\bullet
$$
that, in degree~$q$, classifies the bijection
$$
(F_0; f_1, \dots, f_q) \overset{\cong}\longmapsto (F_0; g_1, \dots, g_q)
$$
where $g_i = f_i \cdots f_1$ for $1 \le i \le q$.  Its inverse is given
by $f_1 = g_1$ and $f_i = g_i g_{i-1}^{-1}$ for $2 \le i \le q$.
\end{proposition}

\begin{proof}
The proof follows the same lines in the split case as for
moving coordinates.  One difference is that in the split case
the isomorphism
$$
L \otimes \beta(B)_q = L \otimes B^{\otimes q}
\cong L \otimes B^{\otimes q} \cong \pi_* THH(MU)_q
$$
can easily be made explicit as the tensor product of the identity on $L$
and an isomorphism $B^{\otimes q} \cong B^{\otimes q}$, given by the
$(q+1-i)$-fold coproduct
$$
B \longto B^{\otimes q+1-i}
$$
from the $i$-th copy of $B$, followed by a permutation and
the $i$-fold product
$$
B^{\otimes i} \longto B
$$
to the $i$-th copy of $B$.
\end{proof}

In the $p$-typical case, the product map
$$
V \otimes V^{\otimes q} = \pi_*(BP) \otimes \pi_*(BP)^{\otimes q}
	\longto \pi_*(BP \wedge BP^{\wedge q}) = \pi_* THH(BP)_q
$$
becomes an isomorphism after rationalization.  Since
$VT$ is flat over~$V$ we can rewrite the target as
$$
\pi_*((BP \wedge BP) \wedge_{BP} \dots \wedge_{BP} (BP \wedge BP))
\cong VT \otimes_V \dots \otimes_V VT \,.
$$
There is an evident isomorphism
$$
VT \otimes_V \dots \otimes_V VT \cong V \otimes T^{\otimes q} \,,
$$
and $V \otimes T^{\otimes q}$ classifies chains
$$
(F_0; \phi_1, \dots, \phi_q) \quad = \quad
(F_0 \overset{\phi_1}\longfrom F_1 \longfrom \dots
	\overset{\phi_q}\longfrom F_q)
$$
of $q$ composable strict isomorphisms between $p$-typical formal group
laws.  On the other hand, $V \otimes \beta(T)_q = V \otimes T^{\otimes
q}$ classifies $q$-tuples
$$
(F_0, \gamma_1, \dots, \gamma_q) \quad = \quad
(F_0 \overset{\gamma_i}\longfrom F_i)_{i=1}^q
$$
of strict isomorphisms between $p$-typical formal group laws, all with
the same target.

\begin{proposition}
There is an isomorphism of simplicial graded commutative rings
$$
V \otimes \beta(T)_\bullet
	\overset{\cong}\longto \pi_* THH(BP)_\bullet
$$
that, in degree~$q$, classifies the bijection
$$
(F_0; \phi_1, \dots, \phi_q)
	\overset{\cong}\longmapsto (F_0; \gamma_1, \dots, \gamma_q)
$$
where $\gamma_i = \phi_1 \cdots \phi_i$ for $1 \le i \le q$.  Its inverse
is given by $\phi_1 = \gamma_1$ and $\phi_i = \gamma_{i-1}^{-1} \gamma_i$
for $2 \le i \le q$.
\end{proposition}

\begin{proof}
The proof is the same as for $\pi_* THH(MU)_\bullet$ with moving
coordinates, except that all formal group laws in sight are $p$-typical.
\end{proof}

\subsection{The skeleton spectral sequence}

We return to $MU$ with moving coordinates.

\begin{proposition} \label{prop:skelTHHMU}
The skeleton spectral sequence for $\pi_* THH(MU)$ collapses at
the $E^2$-term
$$
E^2 = E^\infty \cong L \otimes \Tor^C_*(\bZ, \bZ) \,.
$$
For each $n\ge1$ there is a unique class $\lambda'_n \in \pi_{2n+1}
THH(MU)$ detected by $[c_n]$ in $E^\infty_{1,2n}$, and
$$
\pi_* THH(MU) \cong \pi_*(MU) \otimes E(\lambda'_n \mid n\ge 1) \,.
$$
\end{proposition}

\begin{proof}
The isomorphism of simplicial commutative rings
$$
\pi_* THH(MU)_\bullet \cong L \otimes \beta(C)_\bullet
$$
induces an isomorphism
$$
E^1 = N \pi_* THH(MU)_* \cong L \otimes N \beta(C)_*
$$
of normalized differential graded algebras, hence also of homology
algebras
$$
E^2 \cong L \otimes \Tor^C_*(\bZ, \bZ)
	= L \otimes E([c_n] \mid n\ge1) \,.
$$
The skeleton spectral sequence is a multiplicative first quadrant
spectral sequence.  It follows that it collapses at the $E^2$-term,
since the algebra generators are concentrated in filtrations~$0$ and~$1$.

The class $[c_n] \in E^2_{1,2n} = E^\infty_{1,2n}$ detects
a class $\lambda'_n$ in the image of
$$
\pi_{2n+1}(sk_1 THH(MU)) \longto \pi_{2n+1} THH(MU)
$$
and is well-defined modulo the image of
$$
\pi_{2n+1}(sk_0 THH(MU)) \longto \pi_{2n+1} THH(MU) \,.
$$
Since $\pi_{2n+1}(sk_0 THH(MU)) = \pi_{2n+1}(MU) = 0$, the class
$\lambda'_n \in \pi_{2n+1} THH(MU)$ is in fact well-defined by this
condition.  (We could also have used the fact that $MU$ splits off from
$THH(MU)$, using the augmentation $THH(MU) \to MU$, to arrange that
$\lambda'_n$ maps to zero under the augmentation, but this method of
normalization is irrelevant for the current investigation.)

Since each $\lambda'_n$ is in an odd degree, and $\pi_* THH(MU)$ is
graded-commutative, it follows that the $\lambda'_n$ for $n\ge1$ freely
generate $\pi_* THH(MU)$ over $L \cong \pi_*(MU)$, concluding the proof.
\end{proof}

Here is the split analogue.

\begin{proposition} \label{prop:skelTHHMUsplit}
The skeleton spectral sequence for $\pi_* THH(MU)$ collapses at
the $E^2$-term
$$
E^2 = E^\infty \cong L \otimes \Tor^B_*(\bZ, \bZ) \,.
$$
For each $n\ge1$ there is a unique class $e_n \in \pi_{2n+1}
THH(MU)$ detected by $[b_n]$ in $E^\infty_{1,2n}$, and
$$
\pi_* THH(MU) = L \otimes E(e_n \mid n\ge 1) \,.
$$
\end{proposition}

The expressions~\eqref{eq:cn-xnbn} for the $c_n$ in terms of the absolute
coordinates in $LB$ lead to relations in $L \otimes \Tor^B_*(\bZ, \bZ)$
which detect the following identities in $\pi_* THH(MU)$:
\begin{align*}
\lambda'_1 &= - e_1 \\
\lambda'_2 &= x_1 e_1 - e_2 \\
\lambda'_3 &= (x_2 - x_1^2) e_1 + x_1 e_2 - e_3 \\
\lambda'_4 &= (2 x_3 - 4 x_1 x_2 + x_1^3) e_1 + (x_2 - x_1^2) e_2
	+ x_1 e_3 - e_4 \,.
\end{align*}

Here is the $p$-typical statement.

\begin{proposition} \label{prop:skelTHHBP}
The skeleton spectral sequence for $\pi_* THH(BP)$ collapses at
the $E^2$-term
$$
E^2 = E^\infty \cong V \otimes \Tor^T_*(\bZ_{(p)}, \bZ_{(p)}) \,.
$$
For each $n\ge1$ there is a unique class $\lambda_n \in \pi_{2p^n-1}
THH(BP)$ detected by $[t_n]$ in $E^\infty_{1,2p^n-2}$, and
$$
\pi_* THH(BP) \cong \pi_*(BP) \otimes E(\lambda_n \mid n\ge 1) \,.
$$
\end{proposition}

The proof is the same as for $MU$.

\begin{remark}
Following Andy Baker and Larry Smith, Jim McClure and Ross Staffeldt
\cite{MS93}*{Rem.~4.3} calculated $\pi_* THH(MU) \cong \pi_*(MU)
\otimes E(\lambda'_n \mid n\ge1)$ and $\pi_* THH(BP) \cong \pi_*(BP)
\otimes E(\lambda_n \mid n\ge1)$ as graded rings, where the classes
$\lambda'_n$ in degree~$2n+1$ and $\lambda_n$ in degree~$2p^n-1$ were
only specified in terms of their mod~$p$ Hurewicz images.  Our choices
of generators $\lambda'_n$ and $\lambda_n$ are uniquely defined, and
have the feature that $\sigma(\lambda'_n) = 0$ and $\sigma(\lambda_n)
= 0$.
\end{remark}

\section{The circle action and the right unit}
\label{sec:tcaatru}

For any $S$-algebra $R$, the cyclic structure on $THH(R)_\bullet$
induces a circle action on its realization.  Its restriction
$$
R \wedge S^1_+ \longto THH(R) \wedge S^1_+ \longto THH(R)
$$
to the $0$-skeleton $R = sk_0 THH(R) \subset THH(R)$ factors through the
$1$-skeleton $sk_1 THH(R) \subset THH(R)$ as the map induced by the right
unit $\eta_R \: R \cong S \wedge R \to R \wedge R$, from the pushout
$R \wedge S^1_+ \cong R \wedge (\Delta^1/\partial \Delta^1)_+$ of the maps
$$
R \longfrom R \wedge \partial \Delta^1_+
	\longto R \wedge \Delta^1_+
$$
to the pushout $sk_1 THH(R)$ of the maps
$$
R \longfrom (R \wedge R \wedge \partial \Delta^1_+)
	\cup (R \wedge S \wedge \Delta^1_+)
	\longto R \wedge R \wedge \Delta^1_+ \,.
$$
This follows from the definition of the circle action, which for
a $0$-simplex $x$ traces out the loop given by the $1$-simplex
$t_1 s_0(x)$.
Hence we have a map of horizontal cofiber sequences
$$
\xymatrix{
R \ar[r] \ar[d]_-{=}
	& R \wedge S^1_+ \ar[r] \ar[d]
	& \Sigma R \ar[d] \\
R \ar[r] & sk_1 THH(R) \ar[r] & \Sigma(R \wedge R/S)
}
$$
where the right hand vertical map is the suspension of the composite
$$
R \cong S \wedge R \overset{\eta_R}\longto R \wedge R
	\overset{1\wedge\pi} \longto R \wedge R/S \,.
$$
Using the splitting of
the upper row, we see that the right hand vertical map is
also the composite
$$
\Sigma R \overset{\sigma}\longto sk_1 THH(R) \to \Sigma(R \wedge R/S) \,.
$$
This proves the following result.

\begin{proposition} \label{prop:sigma-etaR}
For $x \in \pi_*(R)$ the homotopy class $\sigma(x) \in \pi_{*+1} THH(R)$
is detected in $E^\infty_{1,*}$ of the skeleton spectral sequence by
the class of the infinite cycle
$$
(1 \wedge \pi) \eta_R(x) \in \pi_*(R \wedge R/S) = E^1_{1,*} \,.
$$
\end{proposition}

We now specialize to $R = MU$.  In terms of moving coordinates, the maps
$\eta_R$ and $\pi$ induce the homomorphisms
$$
L \overset{\eta_R}\longto LC = L \otimes C \overset{1\otimes\pi}\longto
	L \otimes I(C) \,,
$$
where $\eta_R$ is the right unit and $\pi \: C \to I(C)$
is the projection away from $\bZ \to C$.  We can also view $LC \to L
\otimes I(C)$ as the cokernel of the left unit $\eta_L \: L \to LC$.
The split case is practically the same.

\begin{proposition} \label{prop:sigmamn}
The rationalized $\sigma$-operator
$$
\pi_*(MU) \otimes \bQ \longto \pi_{*+1} THH(MU) \otimes \bQ
$$
is the (right) derivation given by
$$
\sigma(m_n) = \lambda'_n \,.
$$
\end{proposition}

\begin{proof}
By Propositions~\ref{prop:etaR-on-mn} and~\ref{prop:sigma-etaR},
$\sigma(m_n)$ is detected in $E^\infty_{1,*} \otimes \bQ$ by the image of
$$
\eta_R(m_n) = \sum_{(i+1)(j+1)=n+1} m_i c_j^{i+1}
$$
under the projections
$$
L \otimes C \otimes \bQ
\longto
L \otimes I(C) \otimes \bQ
\longto
L \otimes I(C)/I(C)^2 \otimes \bQ
= L \otimes \Tor^C_1(\bZ, \bZ) \otimes \bQ \,.
$$
The term with $j=0$ maps to zero in $L \otimes I(C) \otimes \bQ$, and the
terms with $i\ge1$ map to zero in $L \otimes I(C)/I(C)^2 \otimes \bQ$,
so only the term with $i=0$ and $j=n$ remains.  Hence $\sigma(m_n)$
is detected by $[c_n]$ in $E^\infty_{1,*} \otimes \bQ =
L \otimes \Tor^C_1(\bZ, \bZ) \otimes \bQ$, and this characterizes
the homotopy class $\lambda'_n \in \pi_{2n+1} THH(MU)
\subset \pi_{2n+1} THH(MU) \otimes \bQ$.
\end{proof}

\begin{theorem} \label{thm:sigma-Llambdaprimen}
The $\sigma$-operator
$$
\sigma \: \pi_* THH(MU) \longto \pi_{*+1} THH(MU)
$$
is the (right) $\bZ$-linear derivation acting on
$$
\pi_* THH(MU) \cong \pi_*(MU) \otimes E(\lambda'_n \mid n\ge1)
$$
by taking $x \in L \cong \pi_*(MU) \subset \pi_* THH(MU)$ to the
homotopy class $\sigma(x) \in \pi_{*+1} THH(MU)$ detected by the image
of $\eta_R(x) \in LC$ in $L \otimes \Tor^C_1(\bZ, \bZ) = E^\infty_{1,*}$,
while
$$
\sigma(\lambda'_n) = 0
$$
for all $n\ge1$.  In low degrees,
\begin{align*}
\sigma(x_1) &= - 2 \lambda'_1 \\
\sigma(x_2) &= - 4 x_1 \lambda'_1 - 3 \lambda'_2 \\
\sigma(x_3) &= - (4 x_2 + 5 x_1^2) \lambda'_1 - 6 x_1 \lambda_2'
	- 2 \lambda'_3 \\
\sigma(x_4) &= - 4 (2 x_3 - x_1 x_2) \lambda'_1
	- 3 (2 x_2 + x_1^2) \lambda'_2 - 8 x_1 \lambda'_3 - 5 \lambda'_4 \,.
\end{align*}
\end{theorem}

\begin{proof}
The general statements summarize Propositions~\ref{prop:skelTHHMU},
\ref{prop:sigma-etaR} and~\ref{prop:sigmamn}.  We know that $\sigma(\lambda'_n)
= \sigma^2(m_n) = 0$ in $\pi_* THH(MU) \otimes \bQ$, since $\sigma$ acts
as a differential.  Hence $\sigma(\lambda'_n) = 0$ in $\pi_* THH(MU)$,
since these groups are torsion-free.

For the explicit formulas, we first calculate in $\pi_* THH(MU) \otimes
\bQ \cong L \otimes E(\lambda'_n \mid n\ge1) \otimes \bQ$, using the
expressions~\eqref{eq:xn-mn} for the $x_n$ in terms of the $m_n$,
and applying the derivation $\sigma$:
\begin{align*}
\sigma(x_1) &= - 2 \lambda'_1 \\
\sigma(x_2) &= 8 m_1 \lambda'_1 - 3 \lambda'_2 \\
\sigma(x_3) &= 12 (m_2 - 3 m_1^2) \lambda'_1 + 12 m_1 \lambda_2'
	- 2 \lambda'_3 \\
\sigma(x_4) &= 8 (8 m_1^3 - 9 m_1 m_2 + 2 m_3) \lambda'_1
	+ 18 (m_2 - 2 m_1^2) \lambda'_2 + 16 m_1 \lambda'_3 - 5 \lambda'_4 \,.
\end{align*}
The asserted formulas in $\pi_* THH(MU) \cong L \otimes E(\lambda'_n
\mid n\ge1)$ then follow, by rewriting the polynomials in the $m_n$
as elements of $L$.
\end{proof}

Here is the analogous result in the split case.  We do not have a closed
formula for $\sigma(e_n)$.

\begin{theorem} \label{thm:sigma-Len}
The $\sigma$-operator $\sigma \: \pi_* THH(MU) \to \pi_{*+1} THH(MU)$
is the (right) $\bZ$-linear derivation acting on
$$
\pi_* THH(MU) \cong \pi_*(MU) \otimes E(e_n \mid n\ge1)
$$
by taking $x \in L \cong \pi_*(MU) \subset \pi_* THH(MU)$ to the
homotopy class $\sigma(x) \in \pi_{*+1} THH(MU)$ detected by the image of
$\eta_R(x) \in LB$ in $ L \otimes \Tor^B_1(\bZ, \bZ) = E^\infty_{1,*}$.
The classes $\sigma(e_n)$ are inductively determined by the relation
$\sigma^2(x_n) = 0$.  In low degrees,
\begin{align*}
\sigma(x_1) &= 2 e_1 \\
\sigma(x_2) &= x_1 e_1 + 3 e_2 \\
\sigma(x_3) &= (2 x_2 + x_1^2) e_1 + 4 x_1 e_2 + 2 e_3 \\
\sigma(x_4) &= (2 x_1 x_2 - 2 x_3) e_1 + x_2 e_2 + 3 x_1 e_3 + 5 e_4 
\intertext{and}
\sigma(e_1) &= 0 \\
\sigma(e_2) &= 0 \\
\sigma(e_3) &=  e_1 e_2 \\
\sigma(e_4) &= 2 e_1 e_3 \,.
\end{align*}
\end{theorem}

\begin{proof}
For $x_n$ we use~\eqref{eq:etaRxn-xnbn} to calculate
\begin{align*}
(1\otimes\pi)\eta_R(x_1) &= 2 b_1 \\
(1\otimes\pi)\eta_R(x_2) &= x_1 b_1 + (3 b_2 - 2 b_1^2) \\
(1\otimes\pi)\eta_R(x_3) &= (2 x_2 + x_1^2) b_1 +
	x_1 (4 b_2 - b_1^2) + (2 b_3 + 2 b_1 b_2 - 2 b_1^3) \\
(1\otimes\pi)\eta_R(x_4) &=  (2 x_1 x_2 - 2 x_3) b_1 + x_2 (b_2 - b_1^2)
        + x_1 (3 b_3 - 8 b_1 b_2 + 5 b_1^3) \\
&\qquad + (5 b_4 - 14 b_1 b_3 - 6 b_2^2 + 25 b_1^2 b_2 - 10 b_1^4)
\end{align*}
in $E^1_{1,*} = L \otimes I(B)$.  Hence $\sigma(x_n)$ is detected by
\begin{align*}
[(1\otimes\pi)\eta_R(x_1)] &= 2 [b_1] \\
[(1\otimes\pi)\eta_R(x_2)] &= x_1 [b_1] + 3 [b_2] \\
[(1\otimes\pi)\eta_R(x_3)] &= (2 x_2 + x_1^2) [b_1] + 4 x_1 [b_2] + 2 [b_3] \\
[(1\otimes\pi)\eta_R(x_4)] &= (2 x_1 x_2 - 2 x_3) [b_1] + x_2 [b_2]
	+ 3 x_1 [b_3] + 5 [b_4]
\end{align*}
in $E^\infty_{1,*} = L \otimes \Tor^B_1(\bZ, \bZ)$.
Since $e_n \in \pi_{2n+1} THH(MU)$ is characterized by being detected
by $[b_n] \in E^1_{1,2n}$, the stated formulas for $\sigma(x_n)$ hold.
Furthermore, $\sigma^2 = 0$ when acting on $\pi_* THH(MU)$, and the $e_n$
generate an exterior algebra, so it follows that
\begin{align*}
0 &= \sigma(2 e_1) = 2 \sigma(e_1) \\
0 &= \sigma(x_1 e_1 + 3 e_2) = 3 \sigma(e_2) \\
0 &= \sigma((2 x_2 + x_1^2) e_1 + 4 x_1 e_2 + 2 e_3)
	=  -2 e_1 e_2 + 2 \sigma(e_3) \\
0 &= \sigma((2 x_1 x_2 - 2 x_3) e_1 + x_2 e_2 + 3 x_1 e_3 + 5 e_4)
	= - 10 e_1 e_3 + 5 \sigma(e_4) \,.
\end{align*}
Here we have used the form of the Leibniz rule that is appropriate for
right actions, i.e., $\sigma(xy) = x \sigma(y) + (-1)^{|y|} \sigma(x) y$.
Since $\pi_* THH(MU)$ is torsion-free, this implies the stated formulas
for $\sigma(e_n)$.
\end{proof}

We now turn to the $p$-typical case.

\begin{proposition} \label{prop:sigmaln}
The rationalized $\sigma$-operator $\pi_*(BP) \otimes \bQ
\to \pi_{*+1} THH(BP) \otimes \bQ$ is the (right) derivation given by
$$
\sigma(\ell_n) = \lambda_n \,.
$$
\end{proposition}

\begin{proof}
By~\eqref{eq:etaRelln} and Proposition~\ref{prop:sigma-etaR},
$\sigma(\ell_n)$ is detected in $E^\infty_{1,*} \otimes \bQ$ by the image of
$$
\eta_R(\ell_n) = \sum_{i+j=n} \ell_i t_j^{p^i}
$$
under the projections
$$
V \otimes T \otimes \bQ
\longto
V \otimes I(T) \otimes \bQ
\longto
V \otimes I(T)/I(T)^2 \otimes \bQ
= V \otimes \Tor^T_1(\bZ_{(p)}, \bZ_{(p)}) \otimes \bQ \,.
$$
The term with $j=0$ maps to zero in $V \otimes I(T) \otimes \bQ$, and the
terms with $i\ge1$ map to zero in $V \otimes I(T)/I(T)^2 \otimes \bQ$,
so only the term with $i=0$ and $j=n$ remains.  Hence $\sigma(\ell_n)$
is detected by $[t_n]$ in $E^\infty_{1,*} \otimes \bQ =
V \otimes \Tor^T_1(\bZ_{(p)}, \bZ_{(p)}) \otimes \bQ$, and this characterizes
the homotopy class $\lambda_n \in \pi_{2p^n-1} THH(BP)
\subset \pi_{2p^n-1} THH(BP) \otimes \bQ$.
\end{proof}

\begin{theorem} \label{thm:sigma-Vlambdan}
The $\sigma$-operator
$$
\sigma \: \pi_* THH(BP) \longto \pi_{*+1} THH(BP)
$$
is the (right) $\bZ_{(p)}$-linear derivation acting on
$$
\pi_* THH(BP) \cong \pi_*(BP) \otimes E(\lambda_n \mid n\ge1)
$$
by taking $x \in V \cong \pi_*(BP) \subset \pi_* THH(BP)$ to the class
$\sigma(x) \in \pi_{*+1} THH(BP)$ detected by the image of $\eta_R(x)
\in VT$ in $V \otimes \Tor^T_1(\bZ_{(p)}, \bZ_{(p)}) = E^\infty_{1,*}$, while
$$
\sigma(\lambda_n) = 0
$$
for all $n\ge1$.  In low degrees,
\begin{align*}
\sigma(v_1) &= p \lambda_1 \\
\sigma(v_2) &= p \lambda_2 - (p+1) v_1^p \lambda_1 \\
\sigma(v_3) &= p \lambda_3 - (p v_1 v_2^{p-1} + v_1^{p^2}) \lambda_2 \\
	&\qquad - (v_2^p - (p+1) v_1^{p+1} v_2^{p-1}
		+ p^2 v_1^{p^2-1} v_2 + p v_1^{p^2+p}) \lambda_1 \,.
\end{align*}
\end{theorem}

\begin{proof}
The general results are proved as for $MU$ with moving coordinates.
For the explicit formulas, we apply the derivation~$\sigma$
to~\eqref{eq:Hazewinkel}, to obtain
\begin{equation} \label{eq:sigmavn-recursive}
p \lambda_n = \sigma(v_n) + \sum_{i=1}^{n-1} \Bigl(v_{n-i}^{p^i} \lambda_i
	+ (p^i \ell_i) v_{n-i}^{p^i-1} \sigma(v_{n-i}) \Bigr) \,.
\end{equation}
Here $p^i \ell_i$ lies in $V$ by Lemma~\ref{lem:pnellninV}, and
is listed in low degrees in~\eqref{eq:pnelln}.  This leads to
\begin{align*}
p \lambda_1 &= \sigma(v_1) \\
p \lambda_2 &= \sigma(v_2) + (v_1^p \lambda_1 + v_1 v_1^{p-1} \sigma(v_1)) \\
p \lambda_3 &= \sigma(v_3) + (v_2^p \lambda_1 + v_1 v_2^{p-1} \sigma(v_2))
	+ (v_1^{p^2} \lambda_2 + (p v_2 + v_1^{p+1}) v_1^{p^2-1} \sigma(v_1))
\end{align*}
which we can rewrite as stated.
\end{proof}

\section{The circle Tate construction}
\label{sec:tctc}

We can now calculate the $d^2$-differential and $E^3 = E^4$-term of
the circle Tate spectral sequence
$$
E^2_{*,*} = \bZ[t, t^{-1}] \otimes \pi_* THH(MU)
	\Longrightarrow \pi_* THH(MU)^{tS^1} \,,
$$
in the first few degrees.
Since $\eta$ acts trivially on $\pi_* THH(MU)$, the
$d^2$-differential is given by the $\sigma$-operator,
and
$$
E^3_{*,*} = E^4_{*,*}
	= \bZ[t, t^{-1}] \otimes H(\pi_* THH(MU), \sigma) \,.
$$

Let us first note that after rationalization the spectral sequence
collapses after the $d^2$-differential.

\begin{proposition} \label{prop:rationalTHHMUsigma}
Rationally,
\begin{align*}
\pi_*(MU) \otimes \bQ
	&\cong \bQ[m_n \mid n\ge1] \\
\pi_* THH(MU) \otimes \bQ
	&\cong \bQ[m_n \mid n\ge1] \otimes E(\lambda'_n \mid n\ge1) \\
H(\pi_* THH(MU), \sigma) \otimes \bQ
	&\cong \bQ \,.
\end{align*}
\end{proposition}

\begin{proof}
We know that
$\pi_*(MU) \otimes \bQ \cong L \otimes \bQ \cong \bQ[m_n \mid n\ge1]$
and
$$
\pi_* THH(MU) \otimes \bQ
	\cong \pi_*(MU) \otimes E(\lambda'_n \mid n\ge1) \otimes \bQ
	\cong \bQ[m_n \mid n\ge1] \otimes E(\lambda'_n \mid n\ge1) \,,
$$
with $\sigma(m_n) = \lambda'_n$.  Here $H(\bQ[m_n] \otimes E(\lambda'_n),
\sigma) = \bQ$, for each $n\ge1$, so the near-vanishing of $H(\pi_*
THH(MU) \otimes \bQ, \sigma)$ follows by the K{\"u}nneth theorem.
\end{proof}

Integrally, the situation is more complicated.

\begin{theorem} \label{thm:HTHHMUsigma}
$$
H(\pi_* THH(MU), \sigma) = \begin{cases}
\bZ\{1\} & \text{for $*=0$,} \\
0 & \text{for $*=1,2,4,6,8$,} \\
\bZ/2\{\lambda'_1\} & \text{for $*=3$,} \\
\bZ/4\{x_1 \lambda'_1\} \oplus \bZ/3\{\lambda'_2\} & \text{for $*=5$,} \\
\bZ/4\{\lambda'_3\} \oplus \bZ/3\{2 x_1^2 \lambda'_1\} & \text{for $*=7$,} \\
\bZ/16\{(x_3 - 2 x_1 x_2) \lambda'_1 + x_1 \lambda'_3\} & \\
\qquad {} \oplus \bZ/6\{(x_1^2 - x_2) \lambda'_2\} \oplus \bZ/5\{\lambda'_4\}
	& \text{for $*=9$,} \\
\bZ/2\{\lambda'_1 \lambda'_3\} & \text{for $*=10$.}
\end{cases}
$$
\end{theorem}

\begin{proof}
Additively,
\begin{multline*}
\pi_* THH(MU) = (\bZ\{1\}, 0, \bZ\{x_1\}, \bZ\{\lambda'_1\},
        \bZ\{x_1^2, x_2\}, \bZ\{x_1 \lambda'_1, \lambda'_2\}, \\
        \bZ\{x_1^3, x_1 x_2, x_3\},
        \bZ\{x_1^2 \lambda'_1, x_2 \lambda'_1, x_1 \lambda'_2, \lambda'_3\},
        \bZ\{\lambda'_1 \lambda'_2, \dots \}, \dots) \,.
\end{multline*}
As a cochain complex with differential given by the $\sigma$-operator,
this breaks up as a direct sum of the shorter complexes
\begin{align*}
\bZ\{1\} \\
\bZ\{x_1\} &\overset{
  \begin{pmatrix} -2 \end{pmatrix}
  }\longto \bZ\{\lambda'_1\} \\
\bZ\{x_1^2, x_2\} &\overset{
  \begin{pmatrix} -4 & -4 \\ 0 & -3 \end{pmatrix}
  }\longto \bZ\{x_1 \lambda'_1, \lambda'_2\}
\end{align*}
$$
\bZ\{x_1^3, x_1 x_2, x_3\} \overset{
  \begin{pmatrix}
	-6 & -4 & -5 \\
	0 & -2 & -4 \\
	0 & -3 & -6 \\
	0 & 0 & -2
  \end{pmatrix}
  }\longto
\bZ\{x_1^2 \lambda'_1, x_2 \lambda'_1, x_1 \lambda'_2, \lambda'_3\}
  \overset{
  \begin{pmatrix} 0 & -3 & 2 & 0 \end{pmatrix}
  }\longto
\bZ\{\lambda'_1 \lambda'_2\}
$$
and
\begin{multline*}
\bZ\{x_1^4, x_1^2 x_2, x_1 x_3, x_2^2, x_4\} \overset{
  \begin{pmatrix}
	-8 & -4 & -5 & 0 & 0 \\
	0 & -4 & -4 & -8 & 4 \\
	0 & 0 & -2 & 0 & -8 \\
	0 & -3 & -6 & 0 & -3 \\
	0 & 0 & 0 & -6 & -6 \\
	0 & 0 & -2 & 0 & -8 \\
	0 & 0 & 0 & 0 & -5
  \end{pmatrix}
  }\longto \\
\bZ\{x_1^3 \lambda'_1, x_1 x_2 \lambda'_1, x_3 \lambda'_1, x_1^2 \lambda'_2, x_2 \lambda'_2, x_1 \lambda'_3, \lambda'_4\}
  \overset{
  \begin{pmatrix}
	0 & -3 & -6 & 4 & 4 & 0 & 0 \\
	0 & 0 & -2 & 0 & 0 & 2 & 0
  \end{pmatrix}
  }\longto \\
\bZ\{x_1 \lambda'_1 \lambda'_2, \lambda'_1 \lambda'_3\} \,.
\end{multline*}
By rational considerations, $\sigma$ acts injectively on the remaining
summand
$$
\bZ\{x_1^5, x_1^3 x_2, \dots, x_2 x_3, x_5\}
$$
of $\pi_{10} THH(MU)$.  The result then follows by comparing images and
kernels in these complexes.
\end{proof}

Here is the same calculation in absolute coordinates.

\begin{theorem}
$$
H(\pi_* THH(MU), \sigma) = \begin{cases}
\bZ\{1\} & \text{for $*=0$,} \\
0 & \text{for $*=1,2,4,6,8$,} \\
\bZ/2\{e_1\} & \text{for $*=3$,} \\
\bZ/12\{e_2\} & \text{for $*=5$,} \\
\bZ/12\{e'_3\} & \text{for $*=7$,} \\
\bZ/240\{e'_4\} \oplus \bZ/2\{e''_4\} & \text{for $*=9$,} \\
\bZ/2\{e_1 e_3\} & \text{for $*=10$,}
\end{cases}
$$
where
\begin{align*}
e'_3 &= e_3 + 2 x_1 e_2 + x_2 e_1 \\
e'_4 &= e_4 - x_1^2 e_2 - x_3 e_1 \\
e''_4 &= x_1 x_2 e_1 + 3 x_2 e_2 \,.
\end{align*}
\end{theorem}

\begin{proof}
Additively,
\begin{multline*}
\pi_* THH(MU) = (\bZ\{1\}, 0, \bZ\{x_1\}, \bZ\{e_1\},
        \bZ\{x_1^2, x_2\}, \bZ\{x_1 e_1, e_2\}, \\
        \bZ\{x_1^3, x_1 x_2, x_3\},
        \bZ\{x_1^2 e_1, x_2 e_1, x_1 e_2, e_3\},
        \bZ\{e_1 e_2, \dots \}, \dots) \,.
\end{multline*}
The cochain complex $(\pi_* THH(MU), \sigma)$ breaks up as the direct
sum of the complexes
\begin{align*}
\bZ\{1\} \\
\bZ\{x_1\} &\overset{
  \begin{pmatrix} 2 \end{pmatrix}
  }\longto \bZ\{e_1\} \\
\bZ\{x_1^2, x_2\} &\overset{
  \begin{pmatrix} 4 & 1 \\ 0 & 3 \end{pmatrix}
  }\longto \bZ\{x_1 e_1, e_2\} \\
\bZ\{x_1^3, x_1 x_2, x_3\} &\overset{
  \begin{pmatrix} 6 & 1 & 1 \\ 0 & 2 & 2 \\ 0 & 3 & 4 \\ 0 & 0 & 2 \end{pmatrix}
  }\longto
\bZ\{x_1^2 e_1, x_2 e_1, x_1 e_2, e_3\}
  \overset{
  \begin{pmatrix} 0 & 3 & -2 & 1 \end{pmatrix}
  }\longto
\bZ\{e_1 e_2\}
\end{align*}
and
\begin{multline*}
\bZ\{x_1^4, x_1^2 x_2, x_1 x_3, x_2^2, x_4\} \overset{
  \begin{pmatrix}
	8 & 1 & 1 & 0 & 0 \\
	0 & 4 & 2 & 2 & 2 \\
	0 & 0 & 2 & 0 & -2 \\
	0 & 3 & 4 & 0 & 0 \\
	0 & 0 & 0 & 6 & 1 \\
	0 & 0 & 2 & 0 & 3 \\
	0 & 0 & 0 & 0 & 5
  \end{pmatrix}
  }\longto \\
\bZ\{x_1^3 e_1, x_1 x_2 e_1, x_3 e_1, x_1^2 e_2, x_2 e_2, x_1 e_3, e_4\}
  \overset{
  \begin{pmatrix}
	0 & 3 & 4 & -4 & -1 & 1 & 0 \\
	0 & 0 & 2 & 0 & 0 & -2 & 2
  \end{pmatrix}
  }\longto \\
\bZ\{x_1 e_1 e_2, e_1 e_3\} \,.
\end{multline*}
The result then follows by comparing images and kernels.
\end{proof}

\begin{remark}
Ignoring decomposables, one might have expected that the $\sigma$-operator
acts on $\pi_* THH(MU)$ by $\sigma(x_n) = d_n e_n$, where $d_n = p$
if $n+1$ is a power of a prime~$p$ and $d_n = 1$ otherwise, and that
$\sigma(e_n) = 0$.  This would alter the group structure of $H(THH(MU),
\sigma)$ in degree~$7$ to $\bZ/2\{e_3\} \oplus \bZ/6\{x_1^2 e_1\}$,
and in degree~$9$ to $\bZ/120 \oplus (\bZ/2)^2$, and is therefore not
a permissible simplification.
Any expectation that $H(\pi_* THH(MU), \sigma)$ might be trivial
in all positive even degrees, or cyclic in all positive odd degrees, is
also dispelled by these calculations.
\end{remark}

We can also calculate the $d^2$-differential and $E^3 = E^4$-term of
the circle Tate spectral sequence
$$
E^2_{*,*} = \bZ[t, t^{-1}] \otimes \pi_* THH(BP)
	\Longrightarrow \pi_* THH(BP)^{tS^1} \,,
$$
in the first few degrees.  Since $\eta$ acts trivially on $\pi_* THH(BP)$,
the $d^2$-differential is given by the $\sigma$-operator, and
$$
E^3_{*,*} = E^4_{*,*}
	= \bZ[t, t^{-1}] \otimes H(\pi_* THH(BP), \sigma) \,.
$$
Rationally, this spectral sequence collapses after the
$d^2$-differential.

\begin{proposition} \label{prop:rationalTHHBPsigma}
\begin{align*}
\pi_*(BP) \otimes \bQ
	&\cong \bQ[\ell_n \mid n\ge1] \\
\pi_* THH(BP) \otimes \bQ
	&\cong \bQ[\ell_n \mid n\ge1] \otimes E(\lambda_n \mid n\ge1) \\
H(\pi_* THH(BP), \sigma) \otimes \bQ
	&\cong \bQ \,.
\end{align*}
\end{proposition}

The proof is the same as for Proposition~\ref{prop:rationalTHHMUsigma}.

\begin{theorem} \label{thm:HTHHBPsigma}
\begin{multline*}
H(\pi_* THH(BP), \sigma) \\
= \begin{cases}
\bZ_{(p)}\{1\} & \text{for $*=0$,} \\
\bZ/p\{v_1^{i-1} \lambda_1\}
	& \text{for $* = i(2p-2)+1$, $1 \le i \le p-1$,} \\
\bZ/p^2\{v_1^{p-1} \lambda_1\} & \text{for $* = 2p^2 - 2p + 1$,} \\
\bZ/p^2\{\lambda_2\} & \text{for $* = 2p^2 - 1$,} \\
\bZ_{(p)}/p^2(p+2) \{v_2 \lambda_1 + v_1 \lambda_2\}
	& \text{for $* = 2p^2 + 2p - 3$,} \\
\bZ/p\{\lambda_1 \lambda_2\} & \text{for $* = 2p^2 + 2p - 2$,} \\
0 & \text{for the remaining $* \le 2p^2 + 4p - 6$.}
\end{cases}
\end{multline*}
The group in degree~$* = 2p^2 + 2p - 3$ is $\bZ/p^2$ for $p$ odd, and
$\bZ/16$ for $p=2$.
\end{theorem}

\begin{proof}
The cochain complex $(\pi_* THH(BP), \sigma)$ is the direct sum of
a sequence of smaller complexes, which begin with
\begin{align*}
\bZ_{(p)}\{1\} &\\
\bZ_{(p)}\{v_1\} &\overset{ \begin{pmatrix} p \end{pmatrix} }\longto
	\bZ_{(p)}\{\lambda_1\} \\
\bZ_{(p)}\{v_1^2\} &\overset{ \begin{pmatrix} 2p \end{pmatrix} }\longto
	\bZ_{(p)}\{v_1 \lambda_1\} \\
&\vdots \\
\bZ_{(p)}\{v_1^p\} &\overset{ \begin{pmatrix} p^2 \end{pmatrix} }\longto
	\bZ_{(p)}\{v_1^{p-1} \lambda_1\} \\
\bZ_{(p)}\{v_1^{p+1}, v_2\} &\overset{
	\begin{pmatrix} p(p+1) & -(p+1) \\ 0 & p \end{pmatrix} }\longto
	\bZ_{(p)}\{v_1^p \lambda_1, \lambda_2\}
\end{align*}
and
$$
\bZ_{(p)}\{v_1^{p+2}, v_1 v_2\} \overset{
	\begin{pmatrix} p(p+2) & -(p+1) \\ 0 & p \\ 0 & p \end{pmatrix} }\longto
	\bZ_{(p)}\{v_1^{p+1} \lambda_1, v_2 \lambda_1, v_1 \lambda_2\} \overset{
\begin{pmatrix} 0 & p & -p \end{pmatrix} }\longto
	\bZ_{(p)}\{\lambda_1 \lambda_2\} \,.
$$
It is elementary to calculate the cohomology of these complexes.
\end{proof}

\section{Algebraic de Rham cohomology}

For any ring $R$ there is a linearization map $\pi_* THH(R) \to HH_*(R)$
to Hochschild homology, which is a rational isomorphism.  If $R$ is
commutative, then there is also a multiplicative homomorphism $\Omega^*_R
\to HH_*(R)$ from the algebra of de Rham forms to Hochschild homology,
which by the Hochschild--Kostant--Rosenberg theorem \cite{HKR62} is
an isomorphism when $R$ is smooth.  The $\sigma$-operator on $\pi_*
THH(R)$ is compatible with the Connes $B$-operator acting on $HH_*(R)$
and the exterior differential $d$ acting on $\Omega^*_R$, as proved by
Loday--Quillen \cite{LQ84}*{Prop.~2.2}.  Hence the linearization map
from the Tate spectral sequence~\eqref{eq:TPspseq} for $THH(R)$
to the corresponding spectral sequence
\begin{equation} \label{eq:HPspseq}
E^2_{*,*} = \bZ[t, t^{-1}] \otimes HH_*(R)
        \Longrightarrow HP_*(R) \,,
\end{equation}
converging to the periodic cyclic homology $HP_*(R)$, becomes an
isomorphism after rationalization.  In particular, the map of
$E^3$-terms
$$
\bZ[t, t^{-1}] \otimes H(\pi_* THH(R), \sigma)
	\longto \bZ[t, t^{-1}] \otimes H(HH_*(R), B)
$$
is a rational isomorphism.  Furthermore, the induced homomorphism
$$
H_{dR}^*(R) = H(\Omega^*_R, d)
	\longto H(HH_*(R), B)
$$
from the algebraic de Rham cohomology of~$R$ is an isomorphism for $R$
smooth.  It is known \cite{LQ84}*{Thm.~2.9} that after rationalization
the spectral sequence~\eqref{eq:HPspseq} collapses after the
$d^2$-differentials, so that $E^3 \otimes \bQ = E^\infty \otimes \bQ$.

In view of these classical results, it would be interesting to obtain a
more intrinsic algebraic description of the $E^3$-terms $\bZ[t, t^{-1}]
\otimes H(\pi_* THH(MU), \sigma)$ and $\bZ[t, t^{-1}] \otimes H(\pi_*
THH(BP), \sigma)$ of the Tate spectral sequences~\eqref{eq:TPspseq}
converging to $\pi_* THH(MU)^{tS^1}$ and $\pi_* THH(BP)^{tS^1}$,
respectively, than those offered in Theorems~\ref{thm:HTHHMUsigma}
and~\ref{thm:HTHHBPsigma}.
As first approximations to such descriptions we observe below that
there are natural homomorphisms
$$
H^*_{dR}(\pi_*(MU)) \longto
	H(\pi_* THH(MU), \sigma) \longto H^*_{dR}(H_*(MU))
$$
and
$$
H^*_{dR}(\pi_*(BP)) \longto
	H(\pi_* THH(BP), \sigma) \longto H^*_{dR}(H_*(BP)) \,,
$$
relating the de Rham cohomology of the graded commutative rings
$\pi_*(MU) \cong L$, $H_*(MU) \cong C$, $\pi_*(BP) \cong V$ and
$H_*(BP) \cong T$ to the $E^3$-terms of interest.  These are rational
isomorphisms, in a trivial way, but fail to be integral isomorphisms.
Finally, we observe that the Tate spectral sequences~\eqref{eq:TPspseq}
for $MU$ and~$BP$ do not collapse after the $d^2$-differential, due
to the presence of nonzero $d^4$-differentials for $THH(MU)^{tS^1}$
and nonzero $d^{2p}$-differentials for $THH(BP)^{tS^1}$.

\subsection{The Hurewicz homomorphism}
Let $HR$ denote the Eilenberg--Mac\,Lane ring spectrum of a ring~$R$.
There is a unique map $MU \to H\bZ$ of $E_\infty$ ring spectra, and a
unique map $BP \to H\bZ_{(p)}$ of $E_4$ ring spectra.  The following
two lemmas are well-known.

\begin{lemma} \label{lem:HurewiczMU}
There is a commutative diagram of graded commutative rings
$$
\xymatrix{
& B \ar[d]_-{\iota} \ar[dr]^-{\cong} \\
L \ar[r]^-{\eta_R} \ar[d] \ar@(dr,dl)[rr]^-{h}
	& LB \ar[r] & H_*(MU) \ar[d] \\
L \otimes \bQ \ar[rr]^-{h \otimes \bQ}_-{\cong} && H_*(MU) \otimes \bQ
}
$$
where $LB \to H_*(MU)$ is the surjective homomorphism
$$
\pi_*(MU \wedge MU) \longto \pi_*(H\bZ \wedge MU) = H_*(MU)
$$
induced by $MU \to H\bZ$.  The composition $h \: \pi_*(MU) = L \to
H_*(MU)$ is the Hurewicz homomorphism, and $h \otimes \bQ$ sends $m_n \in
L \otimes \bQ$ to the image of $\bar b_n \in B$ in $H_*(MU) \otimes \bQ$.
There is a similar commutative diagram with~$C$ and~$LC$ in place of~$B$
and~$LB$, where $h \otimes \bQ$ sends $m_n$ to the image of $c_n \in C$.
\end{lemma}

\begin{proof}
The Hurewicz homomorphism $h \: \pi_*(MU) \to H_*(MU)$
is induced by the composition
$$
MU \cong S \wedge MU \longto MU \wedge MU \longto H\bZ \wedge MU \,.
$$
The first map induces the right unit $\eta_R \: L \to LB$.  The second
map induces the surjective homomorphism
$$
LB \longto \bZ \otimes_L LB
	\cong \pi_*(H\bZ \wedge_{MU} (MU \wedge MU))
	\cong \pi_*(H\bZ \wedge MU) \,,
$$
where we use that $LB \cong \pi_*(MU \wedge MU)$ is flat as a (left)
module over $L \cong \pi_*(MU)$.  The composition $B \to LB \to
\bZ \otimes_L LB$ is evidently an isomorphism, and similarly for
$C \to LC \to \bZ \otimes_L LC$.

Using~\eqref{eq:etaRmn-in-LB} and Proposition~\ref{prop:etaR-on-mn},
we see that the image of $\eta_R(m_n)$ in $H_*(MU) \otimes \bQ$ is
equal to the images of $\bar b_n \in B$ and $c_n \in C$, since the
remaining terms in each sum are sent to zero under $\pi_*(MU \wedge MU)
\to \pi_*(H\bZ \wedge MU)$.
\end{proof}

\begin{lemma} \label{lem:HurewiczBP}
There is a commutative diagram of graded commutative $\bZ_{(p)}$-algebras
$$
\xymatrix{
& T \ar[d] \ar[dr]^-{\cong} \\
V \ar[r]^-{\eta_R} \ar[d] \ar@(dr,dl)[rr]^-{h}
	& VT \ar[r] & H_*(BP) \ar[d] \\
V \otimes \bQ \ar[rr]^-{h \otimes \bQ}_-{\cong} && H_*(BP) \otimes \bQ
}
$$
where $VT \to H_*(BP)$ is the surjective homomorphism
$$
\pi_*(BP \wedge BP) \longto \pi_*(H\bZ_{(p)} \wedge BP) = H_*(BP)
$$
induced by $BP \to H\bZ_{(p)}$.  The composition $h \: \pi_*(BP) = V \to
H_*(BP)$ is the Hurewicz homomorphism, and $h \otimes \bQ$ sends $\ell_n \in
V \otimes \bQ$ to the image of $t_n \in T$ in $H_*(BP) \otimes \bQ$.
\end{lemma}

\begin{proof}
The proof is similar to that of Lemma~\ref{lem:HurewiczMU}, using
\eqref{eq:etaRelln} to calculate the image of $h(\ell_n)$
in $\bZ_{(p)} \otimes_V VT \otimes \bQ \cong H_*(BP) \otimes \bQ$.
\end{proof}

\subsection{Algebraic de Rham complexes}
The Hurewicz homomorphism $h \: \pi_*(MU) \to H_*(MU)$ maps $\pi_*(MU)
\cong L = \bZ[x_n \mid n\ge1]$ injectively to $H_*(MU) \cong C = \bZ[c_n
\mid n\ge1]$.  Let
$$
\Omega^1_L \cong L \{dx_n \mid n\ge1\}
	\cong L \otimes \Tor^L_1(\bZ, \bZ)
$$
be the module of K{\"a}hler differentials of $L$ over $\bZ$, and let
$\Omega^*_L$ be the algebraic de Rham complex, with $\Omega^q_L \cong L
\otimes \Tor^L_q(\bZ, \bZ)$ in codegree~$q$.  The exterior differential
$d \: \Omega^q_L \to \Omega^{q+1}_L$ is given by $d(x_{n_0} \, dx_{n_1}
\cdots dx_{n_q}) = dx_{n_0} dx_{n_1} \cdots dx_{n_q}$.  Let us view
$$
\pi_* THH(MU) \cong L \otimes \Tor^C_*(\bZ, \bZ)
$$
as a cohomologically graded object (in addition to the internal,
homotopical grading), with $L \otimes \Tor^C_q(\bZ, \bZ)$ in codegree~$q$.
Let $\sigma'(x) = (-1)^{|x|} \sigma(x)$ denote the left derivation
associated to $\sigma$.
We then have inclusions
$$
(\Omega^*_L, d) \longto (\pi_* THH(MU), \sigma') \longto (\Omega^*_C, d)
$$
of cocomplexes, given in codegree~$q$ by
$$
L \otimes \Tor^L_q(\bZ, \bZ) \subset 
L \otimes \Tor^C_q(\bZ, \bZ) \subset 
C \otimes \Tor^C_q(\bZ, \bZ) \,.
$$
The first inclusion maps $dx_n \in \Omega^1_L$ to $\sigma'(x_n) \in \pi_*
THH(MU)$, while the second inclusion maps $\lambda'_n \in \pi_* THH(MU)$
to $dc_n \in \Omega^1_C$, which corresponds to $[c_n] \in \Tor^C_1(\bZ,
\bZ)$.

Similarly, $\pi_*(BP) \cong V = \bZ_{(p)}[v_n \mid n\ge1]$ maps
injectively by the Hurewicz homomorphism to $H_*(BP) \cong T =
\bZ_{(p)}[t_n \mid n\ge1]$.  We view
$$
\pi_* THH(BP) \cong V \otimes \Tor^T_*(\bZ_{(p)}, \bZ_{(p)})
$$
as a cohomologically graded object, with $V \otimes \Tor^T_q(\bZ_{(p)}, \bZ_{(p)})$
in codegree~$q$.
We then have inclusions
$$
(\Omega^*_V, d) \longto (\pi_* THH(BP), \sigma') \longto (\Omega^*_T, d)
$$
of cocomplexes, given in codegree~$q$ by
$$
V \otimes \Tor^V_q(\bZ_{(p)}, \bZ_{(p)}) \subset 
V \otimes \Tor^T_q(\bZ_{(p)}, \bZ_{(p)}) \subset 
T \otimes \Tor^T_q(\bZ_{(p)}, \bZ_{(p)}) \,.
$$
The first inclusion sends $dv_n \in \Omega^1_V$ to $\sigma'(v_n) \in \pi_*
THH(BP)$, while the second inclusion sends $\lambda_n \in \pi_* THH(BP)$
to $dt_n \in \Omega^1_T$, which corresponds to $[t_n] \in \Tor^T_1(\bZ_{(p)},
\bZ_{(p)})$.

Hence $(\pi_* THH(MU), \sigma')$ is bracketed between the de
Rham complexes $(\Omega^*_L, d)$ and $(\Omega^*_C, d)$, while
$(\pi_* THH(BP), \sigma')$ is bracketed between $(\Omega^*_V, d)$ and
$(\Omega^*_T, d)$.  The induced homomorphisms in cohomology
$$
H_{dR}^*(L) = \bigoplus_q H^q(\Omega^*_L, d) \longto
	H(\pi_* THH(MU), \sigma) \longto
	\bigoplus_q H^q(\Omega^*_C, d) = H_{dR}^*(C)
$$
and
$$
H_{dR}^*(V) = \bigoplus_q H^q(\Omega^*_V, d) \longto
        H(\pi_* THH(BP), \sigma) \longto
        \bigoplus_q H^q(\Omega^*_T, d) = H_{dR}^*(T)
$$
are, however, far from isomorphisms.

\subsection{Further differentials}
The $E_\infty$ ring spectrum map $MU \to H\bZ$ induces a homomorphism
$(\pi_* THH(MU), \sigma) \to (\pi_* THH(\bZ), \sigma)$ of
differential graded algebras, sending $\lambda'_1
\in \pi_3 THH(MU) \cong \bZ$ to a generator $g_3 \in \pi_3 THH(\bZ)
\cong \bZ/2$.  In the circle Tate spectral sequence for $THH(\bZ)$
there is a nonzero differential
$$
d^4(t^{-1}) = t g_3 \,,
$$
see \cite{Rog98}*{Thm.~1.3} and \cite{Rog99b}*{Thm.~1.9(2)}.
By naturality, it follows that there is a nonzero differential
$$
d^4(t^{-1}) = \lambda'_1
$$
in the circle Tate spectral sequence for $THH(MU)$.
It also follows that there are nonzero differentials
$$
d^4(t^i \lambda'_3) = t^{i+2} \lambda'_1 \lambda'_3 \,,
$$
for all $i$ of one parity.

Similarly, the $E_4$ ring spectrum map $BP \to H\bZ_{(p)}$ induces a
differential graded algebra homomorphism $(\pi_* THH(BP), \sigma) \to
(\pi_* THH(\bZ_{(p)}), \sigma)$ sending $\lambda_1 \in \pi_{2p-1} THH(BP)
\cong \bZ_{(p)}$ to a generator $g_{2p-1} \in \pi_{2p-1} THH(\bZ_{(p)})
\cong \bZ/p$.  In the circle Tate spectral sequence for $THH(\bZ_{(p)})$
there is a nonzero differential
$$
d^{2p}(t^{1-p}) \doteq t g_{2p-1}
$$
(see \cite{BM94}*{p.~100} in the odd case),
hence there is a nonzero differential
$$
d^{2p}(t^{1-p}) \doteq t \lambda_1
$$
in the circle Tate spectral sequence for $THH(BP)$.
It follows that there are also nonzero differentials
$$
d^{2p}(t^i \lambda_2) \doteq t^{i+p} \lambda_1 \lambda_2 \,,
$$
for $i$ in all but one congruence class of integers modulo~$p$.

These observations show that after the $d^2$-differentials given by the
$\sigma$-operator there will also be later differentials in these
Tate spectral sequences, originating not only on the horizontal axis.
To determine the precise differential structure will require other
methods than those of the present paper.  A good beginning would
be given by determining the differentials in the $C_p$-Tate
spectral sequence
$$
E^2_{*,*} = \hat H^{-*}(C_p, \pi_* THH(MU))
	\Longrightarrow \pi_* THH(MU)^{tC_p} \,,
$$
where we know by \cite{LNR11} that the target is $p$-adically
equivalent to $THH(MU)$.

\affiliationone{%
   John Rognes\\
   Department of Mathematics\\
   University of Oslo\\
   Box 1053, Blindern\\
   NO--0316 Oslo\\
   Norway\\
   \email{rognes@math.uio.no}}


\begin{bibdiv}
\begin{biblist}

\bib{Ada74}{book}{
   author={Adams, J. F.},
   title={Stable homotopy and generalised homology},
   note={Chicago Lectures in Mathematics},
   publisher={University of Chicago Press, Chicago, Ill.-London},
   date={1974},
   pages={x+373},
}

\bib{AR05}{article}{
   author={Angeltveit, Vigleik},
   author={Rognes, John},
   title={Hopf algebra structure on topological Hochschild homology},
   journal={Algebr. Geom. Topol.},
   volume={5},
   date={2005},
   pages={1223--1290},
}

\bib{Ara73}{book}{
   author={Araki, Sh\^{o}r\^{o}},
   title={Typical formal groups in complex cobordism and $K$-theory},
   note={Lectures in Mathematics, Department of Mathematics, Kyoto
   University, No. 6},
   publisher={Kinokuniya Book-Store Co., Ltd., Tokyo},
   date={1973},
   pages={v+101 pp. (loose errata)},
}

\bib{AR02}{article}{
   author={Ausoni, Christian},
   author={Rognes, John},
   title={Algebraic $K$-theory of topological $K$-theory},
   journal={Acta Math.},
   volume={188},
   date={2002},
   number={1},
   pages={1--39},
}

\bib{BM13}{article}{
   author={Basterra, Maria},
   author={Mandell, Michael A.},
   title={The multiplication on BP},
   journal={J. Topol.},
   volume={6},
   date={2013},
   number={2},
   pages={285--310},
}


\bib{BCS10}{article}{
   author={Blumberg, Andrew J.},
   author={Cohen, Ralph L.},
   author={Schlichtkrull, Christian},
   title={Topological Hochschild homology of Thom spectra and the free loop
   space},
   journal={Geom. Topol.},
   volume={14},
   date={2010},
   number={2},
   pages={1165--1242},
}

\bib{BM19}{article}{
   author={Blumberg, Andrew J.},
   author={Mandell, Michael A.},
   title={The homotopy groups of the algebraic $K$-theory of the sphere
   spectrum},
   journal={Geom. Topol.},
   volume={23},
   date={2019},
   number={1},
   pages={101--134},
}

\bib{BBLNR14}{article}{
   author={B\"{o}kstedt, Marcel},
   author={Bruner, Robert R.},
   author={Lun\o e-Nielsen, Sverre},
   author={Rognes, John},
   title={On cyclic fixed points of spectra},
   journal={Math. Z.},
   volume={276},
   date={2014},
   number={1-2},
   pages={81--91},
}

\bib{BHM93}{article}{
   author={B\"{o}kstedt, M.},
   author={Hsiang, W. C.},
   author={Madsen, I.},
   title={The cyclotomic trace and algebraic $K$-theory of spaces},
   journal={Invent. Math.},
   volume={111},
   date={1993},
   number={3},
   pages={465--539},
}

\bib{BM94}{article}{
   author={B\"{o}kstedt, M.},
   author={Madsen, I.},
   title={Topological cyclic homology of the integers},
   note={$K$-theory (Strasbourg, 1992)},
   journal={Ast\'{e}risque},
   number={226},
   date={1994},
   pages={7--8, 57--143},
}

\bib{BM95}{article}{
   author={B\"{o}kstedt, M.},
   author={Madsen, I.},
   title={Algebraic $K$-theory of local number fields: the unramified case},
   conference={
      title={Prospects in topology},
      address={Princeton, NJ},
      date={1994},
   },
   book={
      series={Ann. of Math. Stud.},
      volume={138},
      publisher={Princeton Univ. Press, Princeton, NJ},
   },
   date={1995},
   pages={28--57},
}

\bib{BP66}{article}{
   author={Brown, Edgar H., Jr.},
   author={Peterson, Franklin P.},
   title={A spectrum whose $Z_{p}$ cohomology is the algebra of reduced
   $p^{th}$ powers},
   journal={Topology},
   volume={5},
   date={1966},
   pages={149--154},
}

\bib{BFV07}{article}{
   author={Brun, Morten},
   author={Fiedorowicz, Zbigniew},
   author={Vogt, Rainer M.},
   title={On the multiplicative structure of topological Hochschild
   homology},
   journal={Algebr. Geom. Topol.},
   volume={7},
   date={2007},
   pages={1633--1650},
}

\bib{BR05}{article}{
   author={Bruner, Robert R.},
   author={Rognes, John},
   title={Differentials in the homological homotopy fixed point spectral
   sequence},
   journal={Algebr. Geom. Topol.},
   volume={5},
   date={2005},
   pages={653--690},
}

\bib{Car67a}{article}{
   author={Cartier, Pierre},
   title={Groupes formels associ\'{e}s aux anneaux de Witt g\'{e}n\'{e}ralis\'{e}s},
   language={French},
   journal={C. R. Acad. Sci. Paris S\'{e}r. A-B},
   volume={265},
   date={1967},
   pages={A49--A52},
}

\bib{Car67b}{article}{
   author={Cartier, Pierre},
   title={Modules associ\'{e}s \`a un groupe formel commutatif. Courbes typiques},
   language={French},
   journal={C. R. Acad. Sci. Paris S\'{e}r. A-B},
   volume={265},
   date={1967},
   pages={A129--A132},
}

\bib{DGM13}{book}{
   author={Dundas, Bj\o rn Ian},
   author={Goodwillie, Thomas G.},
   author={McCarthy, Randy},
   title={The local structure of algebraic K-theory},
   series={Algebra and Applications},
   volume={18},
   publisher={Springer-Verlag London, Ltd., London},
   date={2013},
   pages={xvi+435},
}

\bib{DR18}{article}{
   author={Dundas, Bj\o rn Ian},
   author={Rognes, John},
   title={Cubical and cosimplicial descent},
   journal={J. Lond. Math. Soc. (2)},
   volume={98},
   date={2018},
   number={2},
   pages={439--460},
}

\bib{EKMM97}{book}{
   author={Elmendorf, A. D.},
   author={Kriz, I.},
   author={Mandell, M. A.},
   author={May, J. P.},
   title={Rings, modules, and algebras in stable homotopy theory},
   series={Mathematical Surveys and Monographs},
   volume={47},
   note={With an appendix by M. Cole},
   publisher={American Mathematical Society, Providence, RI},
   date={1997},
   pages={xii+249},
}

\bib{GM95}{article}{
   author={Greenlees, J. P. C.},
   author={May, J. P.},
   title={Generalized Tate cohomology},
   journal={Mem. Amer. Math. Soc.},
   volume={113},
   date={1995},
   number={543},
   pages={viii+178},
}

\bib{Haz76}{article}{
   author={Hazewinkel, Michiel},
   title={Constructing formal groups. I. The local one dimensional case},
   journal={J. Pure Appl. Algebra},
   volume={9},
   date={1976/77},
   number={2},
   pages={131--149},
}

\bib{HR}{article}{
   author={Hedenlund, Alice},
   author={Rognes, John},
   title={Multiplicative Tate spectral sequences for compact Lie group
	   actions},
   note={In preparation},
}

\bib{Hes18}{article}{
   author={Hesselholt, Lars},
   title={Topological Hochschild homology and the Hasse-Weil zeta function},
   conference={
      title={An alpine bouquet of algebraic topology},
   },
   book={
      series={Contemp. Math.},
      volume={708},
      publisher={Amer. Math. Soc., Providence, RI},
   },
   date={2018},
   pages={157--180},
}

\bib{HM97}{article}{
   author={Hesselholt, Lars},
   author={Madsen, Ib},
   title={On the $K$-theory of finite algebras over Witt vectors of perfect
   fields},
   journal={Topology},
   volume={36},
   date={1997},
   number={1},
   pages={29--101},
}

\bib{HKR62}{article}{
   author={Hochschild, G.},
   author={Kostant, Bertram},
   author={Rosenberg, Alex},
   title={Differential forms on regular affine algebras},
   journal={Trans. Amer. Math. Soc.},
   volume={102},
   date={1962},
   pages={383--408},
   issn={0002-9947},
   review={\MR{142598}},
   doi={10.2307/1993614},
}

\bib{Lan67}{article}{
   author={Landweber, P. S.},
   title={Cobordism operations and Hopf algebras},
   journal={Trans. Amer. Math. Soc.},
   volume={129},
   date={1967},
   pages={94--110},
}

\bib{Lan73}{article}{
   author={Landweber, Peter S.},
   title={Associated prime ideals and Hopf algebras},
   journal={J. Pure Appl. Algebra},
   volume={3},
   date={1973},
   pages={43--58},
}

\bib{Lan75}{article}{
   author={Landweber, Peter S.},
   title={${\rm BP}\sb\ast({\rm BP})$ and typical formal groups},
   journal={Osaka J. Math.},
   volume={12},
   date={1975},
   number={2},
   pages={357--363},
}

\bib{Laz55}{article}{
   author={Lazard, Michel},
   title={Sur les groupes de Lie formels \`a un param\`etre},
   language={French},
   journal={Bull. Soc. Math. France},
   volume={83},
   date={1955},
   pages={251--274},
}

\bib{LQ84}{article}{
   author={Loday, Jean-Louis},
   author={Quillen, Daniel},
   title={Cyclic homology and the Lie algebra homology of matrices},
   journal={Comment. Math. Helv.},
   volume={59},
   date={1984},
   number={4},
   pages={569--591},
   issn={0010-2571},
   review={\MR{780077}},
   doi={10.1007/BF02566367},
}

\bib{LNR11}{article}{
   author={Lun\o e-Nielsen, Sverre},
   author={Rognes, John},
   title={The Segal conjecture for topological Hochschild homology of
   complex cobordism},
   journal={J. Topol.},
   volume={4},
   date={2011},
   number={3},
   pages={591--622},
}

\bib{LNR12}{article}{
   author={Lun\o e-Nielsen, Sverre},
   author={Rognes, John},
   title={The topological Singer construction},
   journal={Doc. Math.},
   volume={17},
   date={2012},
   pages={861--909},
}

\bib{May77}{book}{
   author={May, J. Peter},
   title={$E_{\infty }$ ring spaces and $E_{\infty }$ ring spectra},
   series={Lecture Notes in Mathematics, Vol. 577},
   note={With contributions by Frank Quinn, Nigel Ray, and J\o rgen
   Tornehave},
   publisher={Springer-Verlag, Berlin-New York},
   date={1977},
   pages={268},
}

\bib{MS93}{article}{
   author={McClure, J. E.},
   author={Staffeldt, R. E.},
   title={On the topological Hochschild homology of $b{\rm u}$. I},
   journal={Amer. J. Math.},
   volume={115},
   date={1993},
   number={1},
   pages={1--45},
}

\bib{Mil75}{book}{
   author={Miller, Haynes Robert},
   title={Some algebraic aspects of the Adams--Novikov spectral sequence},
   note={Thesis (Ph.D.)--Princeton University},
   publisher={ProQuest LLC, Ann Arbor, MI},
   date={1975},
   pages={103},
}

\bib{MRW77}{article}{
   author={Miller, Haynes R.},
   author={Ravenel, Douglas C.},
   author={Wilson, W. Stephen},
   title={Periodic phenomena in the Adams-Novikov spectral sequence},
   journal={Ann. of Math. (2)},
   volume={106},
   date={1977},
   number={3},
   pages={469--516},
}

\bib{Mil60}{article}{
   author={Milnor, J.},
   title={On the cobordism ring $\Omega ^{\ast} $ and a complex analogue.
   I},
   journal={Amer. J. Math.},
   volume={82},
   date={1960},
   pages={505--521},
}

\bib{Mos68}{article}{
   author={Mosher, Robert E.},
   title={Some stable homotopy of complex projective space},
   journal={Topology},
   volume={7},
   date={1968},
   pages={179--193},
}

\bib{NS18}{article}{
   author={Nikolaus, Thomas},
   author={Scholze, Peter},
   title={On topological cyclic homology},
   journal={Acta Math.},
   volume={221},
   date={2018},
   number={2},
   pages={203--409},
}

\bib{Nov60}{article}{
   author={Novikov, S. P.},
   title={Some problems in the topology of manifolds connected with the
   theory of Thom spaces},
   journal={Soviet Math. Dokl.},
   volume={1},
   date={1960},
   pages={717--720},
}

\bib{Nov67}{article}{
   author={Novikov, S. P.},
   title={Rings of operations and spectral sequences of Adams type in
   extraordinary cohomology theories. $U$-cobordism and $K$-theory},
   language={Russian},
   journal={Dokl. Akad. Nauk SSSR},
   volume={172},
   date={1967},
   pages={33--36},
}

\bib{Qui69}{article}{
   author={Quillen, Daniel},
   title={On the formal group laws of unoriented and  complex cobordism
   theory},
   journal={Bull. Amer. Math. Soc.},
   volume={75},
   date={1969},
   pages={1293--1298},
}

\bib{Rav86}{book}{
   author={Ravenel, Douglas C.},
   title={Complex cobordism and stable homotopy groups of spheres},
   series={Pure and Applied Mathematics},
   volume={121},
   publisher={Academic Press, Inc., Orlando, FL},
   date={1986},
   pages={xx+413},
}

\bib{Rav92}{book}{
   author={Ravenel, Douglas C.},
   title={Nilpotence and periodicity in stable homotopy theory},
   series={Annals of Mathematics Studies},
   volume={128},
   note={Appendix C by Jeff Smith},
   publisher={Princeton University Press, Princeton, NJ},
   date={1992},
   pages={xiv+209},
}

\bib{Rog98}{article}{
   author={Rognes, John},
   title={Trace maps from the algebraic $K$-theory of the integers (after
   Marcel B\"{o}kstedt)},
   journal={J. Pure Appl. Algebra},
   volume={125},
   date={1998},
   number={1-3},
   pages={277--286},
}

\bib{Rog99a}{article}{
   author={Rognes, John},
   title={The product on topological Hochschild homology of the integers
   with mod $4$ coefficients},
   journal={J. Pure Appl. Algebra},
   volume={134},
   date={1999},
   number={3},
   pages={211--218},
}

\bib{Rog99b}{article}{
   author={Rognes, John},
   title={Topological cyclic homology of the integers at two},
   journal={J. Pure Appl. Algebra},
   volume={134},
   date={1999},
   number={3},
   pages={219--286},
}

\bib{Rog99c}{article}{
   author={Rognes, John},
   title={Algebraic $K$-theory of the two-adic integers},
   journal={J. Pure Appl. Algebra},
   volume={134},
   date={1999},
   number={3},
   pages={287--326},
}

\bib{Rog02}{article}{
   author={Rognes, John},
   title={Two-primary algebraic $K$-theory of pointed spaces},
   journal={Topology},
   volume={41},
   date={2002},
   number={5},
   pages={873--926},
}

\bib{Rog03}{article}{
   author={Rognes, John},
   title={The smooth Whitehead spectrum of a point at odd regular primes},
   journal={Geom. Topol.},
   volume={7},
   date={2003},
   pages={155--184},
}

\end{biblist}
\end{bibdiv}
\end{document}